\documentclass[a4paper,11pt]{article}
\usepackage[hypertexnames=false]{hyperref}
\usepackage{amsmath, amssymb, amsthm}
\usepackage{enumerate}
\usepackage[square, comma, numbers, sort&compress]{natbib}
\usepackage{color,setspace}
\usepackage{cancel}
\usepackage{todonotes}
\usepackage{dsfont}
\usepackage[a4paper,margin=1in]{geometry}
\usepackage{algorithm}
\usepackage{algpseudocode}
\usepackage{graphicx}
\usepackage{float}
\usepackage{subcaption}

\allowdisplaybreaks

\renewcommand{\tilde}{\widetilde}

\theoremstyle{plain}
\newtheorem{theorem}{Theorem}[section]
\newtheorem{lemma}[theorem]{Lemma}

\theoremstyle{definition}
\newtheorem{example}{Example}[section]
\newtheorem{remark}[theorem]{Remark}
\newtheorem{assumption}{Assumption}[section]
\numberwithin{equation}{section}
\numberwithin{algorithm}{section}
\def\theequation{\arabic{section}.\arabic{equation}}

\input{tcilatex}
\begin{document}

\title{Random walk algorithm for the Dirichlet problem for parabolic
integro-differential equation}
\author{G. Deligiannidis\thanks{%
Department of Statistics, University of Oxford} \and S. Maurer\thanks{%
School of Mathematical Sciences, University of Nottingham, UK} \and M.V.
Tretyakov\thanks{%
School of Mathematical Sciences, University of Nottingham, UK;
Michael.Tretyakov@nottingham.ac.uk}}
\maketitle

\begin{abstract}
We consider stochastic differential equations driven by a general L\'evy
processes (SDEs) with infinite activity and the related, via the Feynman-Kac
formula, Dirichlet problem for parabolic integro-differential equation
(PIDE). We approximate the solution of PIDE using a numerical method for the
SDEs. The method is based on three ingredients: (i) we approximate small
jumps by a diffusion; (ii) we use restricted jump-adaptive time-stepping;
and (iii) between the jumps we exploit a weak Euler approximation. We prove
weak convergence of the considered algorithm and present an in-depth
analysis of how its error and computational cost depend on the jump activity
level. Results of some numerical experiments, including pricing of barrier
basket currency options, are presented.


\end{abstract}

\section{Introduction}

Stochastic differential equations driven by L\'evy processes (SDEs) have
become a very important modelling tool in finance, physics, and biology (see
e.g. \cite%
{ContTankov,van1995stochastic,barndorff2012levy,allen2010introduction}).
Successful use of SDEs relies on effective numerical methods. In this paper,
we are interested in weak-sense approximation of SDEs driven by general
L\'evy processes in which the noise has both the Wiener process and Poisson
processes components including the case of infinite jump activity.

Let $G$ be a bounded domain in $\mathbb{R}^{d}$, $Q=[t_{0},T)\times G$ be a
cylinder in $\mathbb{R}^{d+1},$ $\Gamma =\bar{Q}\setminus Q$ be the part of
the cylinder's boundary consisting of the upper base and lateral surface, $%
G^{c}=\mathbb{R}^{d} \setminus Q$ be the complement of $G$ and $%
Q^{c}:=(t_{0},T]\times G^{c}\cup \{T\}\times \bar{G}.$ Consider the
Dirichlet problem for the parabolic integro-differential equation (PIDE):%
\begin{equation}  \label{eq:problem}
\begin{split}
\frac{\partial u}{\partial t}+Lu+c(t,x)u+g(t,x)& =0, \quad (t,x)\in Q, \\
u(t,x)& =\varphi (t,x), \quad (t,x)\in Q^{c},
\end{split}%
\end{equation}%
where the integro-differential operator $L$ is of the form
\begin{gather}
Lu(t,x):=\frac{1}{2}\sum_{i,j=1}^{d}a^{ij}(t,x)\frac{\partial ^{2}u}{%
\partial x^{i}\partial x^{j}}(t,x)+\sum_{i=1}^{d}b^{i}(t,x)\frac{\partial u}{%
\partial x^{i}}(t,x)  \label{eq:operator} \\
\qquad \qquad +\int_{\mathbb{R}^{m}}\Big\{u\big(t,x+F(t,x)z\big)%
-u(t,x)-\langle F(t,x)z,\nabla u(t,x)\rangle \mathbf{I}_{|z|\leq 1}\Big\}\nu
(\mathrm{d}z);  \notag
\end{gather}%
$a(t,x)=\left( a^{ij}(t,x)\right) $ is a $d\times d$-matrix; $%
b(t,x)=(b^{1}(t,x),\ldots ,b^{d}(t,x))^{\top }$ is a $d$-dimensional vector;
$c(t,x),$ $g(t,x),$ and $\varphi (t,x)$ are scalar functions; $F(t,x)=\left(
F^{ij}(t,x)\right) $ is a $d\times m$-matrix; and $\nu (z),$ $z\in \mathbb{R}%
^{m},$ is a L\'evy measure such that $\int_{\mathbb{R}^{m}}(|z|^{2}\wedge
1)\nu (\mathrm{d}z)<\infty .$ We allow $\nu$ to be of infinite intensity,
i.e. we may have $\nu \big(B(0,r)\big)=\infty $ for some $r>0$, where as
usual for $x\in \mathbb{R}^d$ and $s>0$ we write $B(x,s)$ for the open ball
of radius $s$ centred at $x$.

The Feynman-Kac formula provides a probabilistic representations of the
solution $u(t,x)$ to (\ref{eq:problem}) in terms of a system of
L\'evy-driven SDEs (see Section~\ref{sec_pre}), which can be viewed as a
system of characteristics for this PIDE. A weak-sense approximation of the
SDEs together with the Monte Carlo technique gives us a numerical approach
to evaluating $u(t,x)$, which is especially effective in higher dimensions.

There has been a considerable amount of research on weak-sense numerical
methods for L\'evy-type SDEs of finite and infinite activity (see e.g. \cite%
{Mik88,Protter97,Liu2000,jacod2005approximate,rubenthaler2003numerical,
mordecki2008adaptive,Higa10,Platen2010,KOT13,Mik18} and references therein).
Our approach is most closely related to \cite{KOT13}. As in \cite%
{Asmus01,Higa10,KOT13}, we replace small jumps with an appropriate Brownian
motion, which makes the numerical solution of SDEs with infinite activity of
the L\'evy measure feasible in practice. There are three main differences
between our approach and that of \cite{KOT13}. \textit{First}, we use
restricted jump-adapted time-stepping while in \cite{KOT13} jump-adapted
time-stepping was used. Here by jump-adapted we mean that time
discretization points are located at jump times $\tau _{k}$ and between the
jumps the remaining diffusion process is effectively approximated \cite%
{Higa10,KOT13}. By restricted jump-adapted time-stepping, we understand the
following. We fix a time-discretization step $h>0$. If the jump time
increment $\delta $ for the next time step is less than $h,$ we set the time
increment $\theta =\delta ,$ otherwise $\theta =h,$ i.e., our time steps are
defined as $\theta =\delta \wedge h.$ We note that this is a different
time-stepping strategy to commonly used ones in the literature including the
finite-activity case (i.e., jump-diffusion). For example, in the finite
activity case it is common \cite{Liu2000, mordecki2008adaptive,Platen2010}
to simulate $\tau _{k}$ before the start of simulations and then superimpose
those random times on a grid with some constant or variable finite, small
time-step $h$. Our time-stepping approach is more natural for the problem
under consideration than both commonly used strategies; its benefits are
discussed in Section~\ref{sec:weak}, with the infinite activity case
discussed in more detail in Subsections~\ref{sec:iij} and~\ref%
{ssec:exa-singular}. Restricting $\delta $ by $h$ is beneficial for accuracy
when jumps are rare (e.g. in the jump-diffusion case) and it is also
beneficial for convergence rates (measured in the average number of steps)
in the case of $\alpha $-stable L\'evy measure with $\alpha \in (1,2)$ (see
Sections~\ref{sec:weak} and~\ref{sec:num}). \textit{Second}, in comparison
with \cite{KOT13} we explicitly show (singular) dependence of the numerical
integration error of our algorithm on the parameter $\epsilon $ which is the
cut-off for small jumps replaced by the Brownian motion. \textit{Third, }in
comparison with the literature we consider the Dirichlet problem for PIDEs,
though we also comment on the Cauchy case in Subsection~\ref{sec:Cau}, which
is novel with respect to the use of restricted time-stepping and dependence
of the algorithm's error on $\epsilon $.

The paper is organised as follows. In Section~\ref{sec_pre}, we write down a
probabilistic representation for the solution $u(t,x)\ $of (\ref{eq:problem}%
), we state assumptions used throughout the paper, and we consider the
approximation $u^{\epsilon }(t,x)$ that solves an auxiliary Dirichlet
problem corresponding to the system of characteristics with jumps cut-off by
$\epsilon$. In Section~\ref{sec:weak}, we introduce the numerical algorithm
which approximates $u^{\epsilon }(t,x).$ The algorithm uses the restricted
jump-adapted time-stepping and approximates the diffusion by a weak Euler
scheme. In this section we also obtain and discuss the weak-sense error
estimate for the algorithm. In Section~\ref{sec:num}, we illustrate our
theoretical findings by three numerical examples, including an application
of our algorithm to pricing an FX barrier basket option whose underlyings
follow an exponential L\'evy model.

\section{Preliminaries\label{sec_pre}}

Let $(\Omega ,\mathcal{F},\left\{ \mathcal{F}_{t}\right\} _{t_{0}\leq t\leq
T},P)$ be a filtered probability space satisfying the usual hypotheses. The
operator $L$ defined in \eqref{eq:operator}, on an appropriate domain, is
the generator of the $d$-dimensional process $X_{t_{0},x}(t)$ given by
\begin{equation}
X_{t_{0},x}(t)=x+\int_{t_{0}}^{t}b(s,X(s-))\mathrm{d}s+\int_{t_{0}}^{t}%
\sigma (s,X(s-))\mathrm{d}w(s)+\int_{t_{0}}^{t}\int_{\mathbb{R}%
^{d}}F(s,X(s-))z\hat{N}(\mathrm{d}z,\mathrm{d}s),  \label{eq:generalsde}
\end{equation}%
where the $d\times d$ matrix $\sigma (s,x)$ is defined through $\sigma
(s,x)\sigma ^{\top }(s,x)=a(s,x);$ $w(t)=(w^{1}(t),\ldots ,w^{d}(t))^{\top }$
is a standard $d$-dimensional Wiener process; and $\hat{N}$ is a Poisson
random measure on $[0,\infty )\times \mathbb{R}^{m}$ with intensity measure $%
\nu (\mathrm{d}z)\times \mathrm{d}s$, $\int_{\mathbb{R}^{m}}(|z|^{2}\wedge
1)\nu (\mathrm{d}z)<\infty ,$ and compensated small jumps, i.e.,
\begin{equation*}
\hat{N}\left( [0,t]\times B\right) =\int_{[0,t]\times B}N(\mathrm{d}z,%
\mathrm{d}s)-t\nu (B\cap \{|z|\leq 1\}),\quad
\mbox{for all $t\geq 0$ and $B\in
\mathcal{B}\big(\mathbb{R}^m \big)$}.
\end{equation*}

\begin{remark}
Often \cite{Apple09,Protter97} a simpler model of the form
\begin{equation}
X(t)=x+\int_{t_{0}}^{t}F(s,X(s-))\mathrm{d}Z(s),  \label{eq:simplesde}
\end{equation}%
where $Z(t),$ $t\geq t_{0},$ is an $m$-dimensional L\'{e}vy process with the
characteristic exponent%
\begin{equation*}
\psi (\xi )=\mathrm{i}(\mu ,\xi )-\frac{1}{2}(\xi ,\sigma \xi
)+\int_{|z|\leq 1}\Big[\mathrm{e}^{\mathrm{i}(\xi ,z)}-1-\mathrm{i}(\xi ,z)%
\Big]\nu (\mathrm{d}z)+\int_{|z|>1}\Big[\mathrm{e}^{\mathrm{i}(\xi ,z)}-1%
\Big]\nu (\mathrm{d}z),
\end{equation*}%
is considered instead of the general SDEs \eqref{eq:generalsde}. The
equation (\ref{eq:simplesde}) is obtained as a special case of %
\eqref{eq:generalsde} by setting $b(t,x)=\mu F(t,x)$ and $\sigma
(t,x)=\sigma F(t,x)$.
\end{remark}

%
%

When the solution $u$ of \eqref{eq:problem} is regular enough, for example $%
u\in C^{1,2}\left( [t_{0},T]\times \mathbb{R}^{d}\right) $, it can be shown,
see e.g. \cite{Apple09}, that $u$ has the following probabilistic
representation
\begin{equation}
u(t,x)=\mathbb{E}\left[ \varphi \left( \tau _{t,x},X_{t,x}(\tau
_{t,x})\right) Y_{t,x,1}(\tau _{t,x})+Z_{t,x,1,0}(\tau _{t,x})\right] ,\ \
(t,x)\in Q,  \label{eq:fkac}
\end{equation}%
where $(X_{t,x}(s),Y_{t,x,y}(s),Z_{t,x,y,z}(s))$ for $s\geq t$, solves the
system of SDEs consisting of (\ref{eq:generalsde}) and
\begin{eqnarray}
dY &=&c(s,X(s-))Yds,\ \ Y_{t,x,y}(t)=y,  \label{fkacy} \\
dZ &=&g(s,X(s-))Yds,\ \ Z_{t,x,y,z}(t)=z,  \label{fkatz}
\end{eqnarray}%
and $\tau _{t,x}=\inf \{s\geq t:(s,X_{t,x}(s))\notin Q\}$ is the fist
exit-time of the space-time L\'{e}vy process $(s,X_{t,x}(s))$ from the
space-time cylinder $Q$.

If one can simulate trajectories of $%
\{(s,X_{t,x}(s),Y_{t,x,1}(s),Z_{t,x,1,0}(s));s\geq 0\}$ then the solution of
the Dirichlet problem for PIDE \eqref{eq:problem} can be estimated by
applying the Monte Carlo technique to \eqref{eq:fkac}. This approach however
is not generally implementable for L\'{e}vy measures of infinite intensity,
that is when $\nu \big(B(0,r)\big)=\infty $ for some $r>0$. The difficulty
arises from the presence of an infinite number of small jumps in any finite
time interval, and can be overcome by replacing these small jumps by an
appropriate diffusion exploiting the idea of the method developed in \cite%
{Higa10,Asmus01}, which we apply here. {Alternatively, the issue can be
overcome if one can simulate directly from the increments of L\'{e}vy
process. We will not discuss this case in this paper as we only assume that
one has access to the L\`{e}vy measure. }

\subsection{Approximation of small jumps by diffusion\label{sec_sjd}}

We will now consider the approximation of (\ref{eq:generalsde}) discussed
above, where small jumps are replaced by an appropriate diffusion. In the
case of the whole space (the Cauchy problem for a PIDE) such an
approximation was considered in \cite{Higa10,Asmus01}.

Let $\gamma _{\epsilon }$ be an $m$-dimensional vector with the components
\begin{equation}
\gamma _{\epsilon }^{i}=\int_{\epsilon \leq |z|\leq 1}z^{i}\nu (\mathrm{d}z);
\label{eq:gen21}
\end{equation}%
and $B_{\epsilon }$ is an $m\times m$ matrix with the components
\begin{equation}
B_{\epsilon }^{ij}=\int_{|z|<\epsilon }z^{i}z^{j}\nu (\mathrm{d}z),
\label{eq:gen22}
\end{equation}%
while $\beta _{\epsilon }$ be obtained from the formula $\beta _{\epsilon
}\beta _{\epsilon }^{\top }=B_{\epsilon }.$

\begin{remark}
In many practical situations (see e.g. \cite{ContTankov}), where the
dependence among the components of $X(t)$ introduced through the structure
of the SDEs is enough, we can allow the components of the driving Poisson
measure to be independent. This amounts to saying that $\nu $ is
concentrated on the axes, and as a result $B_{\epsilon }$ will be a diagonal
matrix.
\end{remark}

We shall consider the modified jump-diffusion $\tilde{X}_{t_{0},x}(t)=\tilde{%
X}_{t_{0},x}^{\epsilon }(t)$ defined as
\begin{align}
\tilde{X}_{t_{0},x}(t)& =x+\int_{t_{0}}^{t}\left[ b(s,\tilde{X}(s-))-F(s,%
\tilde{X}(s-))\gamma _{\epsilon }\right] \mathrm{d}s+\int_{t_{0}}^{t}\sigma
(s,\tilde{X}(s-))\mathrm{d}w(s)  \label{eq:gensde2} \\
& \quad +\int_{t_{0}}^{t}F(s,\tilde{X}(s-))\beta _{\epsilon }\mathrm{d}%
W(s)+\int_{t_{0}}^{t}\int_{|z|\geq \epsilon }F(s,\tilde{X}(s-))zN(\mathrm{d}%
z,\mathrm{d}s),  \notag
\end{align}%
where $W(t)$ is a standard $m$-dimensional Wiener process, independent of $N$
and $w$. We observe that, in comparison with \eqref{eq:generalsde}, in %
\eqref{eq:gensde2} jumps less than $\epsilon $ in magnitude are replaced by
the additional diffusion part. In this way, the new L\'{e}vy measure has
finite activity allowing us to simulate its events exactly, i.e. in a
practical way.

Consequently, we can approximate the solution of $u(t,x)$ the PIDE %
\eqref{eq:problem} by
\begin{equation}
u(t,x)\approx u_{\epsilon }(t,x):=\mathbb{E}\left[ \varphi \left( \tilde{\tau%
}_{t,x},\tilde{X}_{t,x}(\tilde{\tau}_{t,x})\right) \tilde{Y}_{t,x,1}(\tilde{%
\tau}_{t,x})+\tilde{Z}_{t,x,1,0}(\tilde{\tau}_{t,x})\right] ,\ \ (t,x)\in Q,
\label{eq:fkac2}
\end{equation}%
where $\tilde{\tau}_{t,x}=\inf \{s\geq t:(s,\tilde{X}_{t,x}(s))\notin Q\}$
is the fist exit time of the space-time L\'evy process $(s,\tilde{X}%
_{t,x}(s))$ from the space-time cylinder $Q$ and $\left( \tilde{X}_{t,x}(s),%
\tilde{Y}_{t,x,y}(s),\tilde{Z}_{t,x,y,z}(s)\right) _{s\geq 0}$ solves the
system of SDEs consisting of \eqref{eq:gensde2} along with
\begin{eqnarray}
d\tilde{Y} &=&c(s,\tilde{X}(s-))\tilde{Y}ds,\ \ \tilde{Y}_{t,x,y}(t)=y,
\label{eq:fkac2y} \\
d\tilde{Z} &=&g(s,\tilde{X}(s-))\tilde{Y}ds,\ \ \tilde{Z}_{t,x,y,z}(t)=z.
\label{eq:fkac2z}
\end{eqnarray}%
Since the new L\'evy measure has finite activity, we can derive a
constructive weak scheme for (\ref{eq:gensde2}), (\ref{eq:fkac2y})-(\ref%
{eq:fkac2z}) (see Section~\ref{sec:weak}). By using this method together
with the Monte Carlo technique, we will arrive at an implementable
approximation of $u_{\epsilon }(t,x)$ and hence of $u(t,x).$

We will next show that indeed $u_{\epsilon }$ defined in \eqref{eq:fkac2} is
a good approximation to the solution of \eqref{eq:problem}. Before
proceeding, we need to formulate appropriate assumptions.

\subsection{Assumptions}

\label{ssec:assumptions}

First, we make the following assumptions on the coefficients of the problem %
\eqref{eq:problem} which will guarantee, see e.g. \cite{Apple09}, that the
SDEs \eqref{eq:generalsde}, \eqref{fkacy}-\eqref{fkatz} and %
\eqref{eq:gensde2}, \eqref{eq:fkac2y}-\eqref{eq:fkac2z} have unique adapted,
c\`adl\`ag solutions with finite moments.

\begin{assumption}
\label{assumption3.1}(\textit{Lipschitz condition}) \textit{There exists a
constant} $K>0$ \textit{such that for all} $x_{1},$ $x_{2}\in \mathbb{R}^{d}$
\textit{and all} $t\in \lbrack t_{0},T]$,
\begin{multline}  \label{eq:lipschitz}
\big|b(t,x_{1})-b(t,x_{2})\big|^{2}+\big\|\sigma (t,x_{1})-\sigma (t,x_{2})%
\big\|^{2} +\left\vert c(t,x_{1})-c(t,x_{2})\right\vert ^{2}+\left\vert
g(t,x_{1})-g(t,x_{2})\right\vert ^{2} \\
+\int_{\mathbb{R}^{d}}\Vert F(t,x_{1})-F(t,x_{2})\Vert ^{2}|z|^{2}\nu (%
\mathrm{d}z)\leq K|x_{1}-x_{2}|^{2}.
\end{multline}
\end{assumption}

\begin{assumption}
\label{assumption3.2}(\textit{Growth condition}) \textit{There exists a
constant} $K>0$ \textit{such that for all} $x\in \mathbb{R}^{d}$ \textit{and
all} $t\in \lbrack t_{0},T]$,
\begin{align}
\big|b(t,x)\big|^{2}+\big\|\sigma (t,x)\big\|^{2}+\left\vert
g(t,x)\right\vert ^{2}+\int_{\mathbb{R}^{d}}\Vert F(t,x)\Vert ^{2}|z|^{2}\nu
(\mathrm{d}z)&\leq K(1+|x|)^{2},  \label{eq:growth} \\
|c(t,x)|&\leq K.  \label{eq:growth2}
\end{align}
\end{assumption}

In order to streamline the presentation and avoid lengthy technical
discussions (see Remark~\ref{rem:regu}), we will make the following
assumption regarding the regularity of solutions to (\ref{eq:problem}).

\begin{assumption}
\label{assumption2.1} \textit{The Dirichlet problem }\eqref{eq:problem}
\textit{admits a classical solution} $u(\cdot ,\cdot )\in
C^{l,m}([t_{0},T]\times \mathbb{R}^{d})$ \textit{with some }$l\geq 1$
\textit{and} $m\geq 2$.
\end{assumption}

In addition to the PIDE problem \eqref{eq:problem}, we also consider the
PIDE problem for $u^{\epsilon }$ from \eqref{eq:fkac2} \cite{Apple09}:%
\begin{gather}
\frac{\partial u^{\epsilon }}{\partial t}+L_{\epsilon }u^{\epsilon
}+c(t,x)u^{\epsilon }+g(t,x) =0, \,\,\, (t,x) \in Q,  \label{v1} \\
u^{\epsilon }(t,x) =\varphi (t,x), \,\,\, (t,x) \in Q^{c},  \notag
\end{gather}%
where
\begin{gather}
L_{\epsilon }v(t,x):=\frac{1}{2}\sum_{i,j=1}^{d}\left[ a^{ij}(t,x)+\left(
F(t,x)B_{\epsilon }(t,x)F^{\top }(t,x)\right) ^{ij}\right] \frac{\partial
^{2}v}{\partial x^{i}\partial x^{j}}(t,x)  \label{v2} \\
+\sum_{i=1}^{d}\Big(b^{i}(t,x)-\sum_{j=1}^{m}F^{ij}(t,x)\gamma _{\epsilon
}^{j}\Big)\frac{\partial v}{\partial x^{i}}(t,x)+\int_{|z|\geq \epsilon }%
\Big\{v\big(t,x+F(t,x)z\big)-v(t,x)\Big\}\nu (\mathrm{d}z).  \notag
\end{gather}

Again, for simplicity (but see Remark~\ref{rem:regu}), we impose the
following conditions on the solution $u_{\epsilon }$ of the above Dirichlet
problem.

\begin{assumption}
\label{assumption2.1'} \textit{The auxiliary Dirichlet problem} \eqref{v1}
\textit{admits a classical solution }$u^{\epsilon }(\cdot ,\cdot )\in
C^{l,m}([t_{0},T]\times \mathbb{R}^{d})$ \textit{with some} $l\geq 1$
\textit{and} $m\geq 2$.
\end{assumption}

Finally, we also require that $u^{\epsilon }$ and its derivatives do not
grow faster than a polynomial function at infinity.

\begin{assumption}[\textit{Smoothness and growth}]
\label{assum:differentiability} There exist constants $K>0$ and $q\geq 1$
such that for all $x\in \mathbb{R}^{d}$, all $t\in \lbrack t_{0},T]$ and $%
\epsilon >0$, the solution $u^{\epsilon }$ of the PIDE problem \eqref{v1}
and its derivatives satisfy
\begin{equation}
\Big|\frac{\partial ^{l+j}}{\partial t^{l}\partial x^{i_{1}}\cdots \partial
x^{i_{j}}}u^{\epsilon }(t,x)\Big|\leq K(1+|x|^{q}),  \label{eq:dergrowth}
\end{equation}%
where $0\leq 2l+j\leq 4,\ \sum_{k=1}^{j}i_{k}=j,$ and $i_{k}$ are integers
from $0$ to $j$.
\end{assumption}

\begin{remark}
\label{rem:regu} Sufficient conditions guaranteeing Assumptions~\ref%
{assumption2.1}, \ref{assumption2.1'} and ~\ref{assum:differentiability}
consist in sufficient smoothness of the coefficients, the boundary $\partial
G,$ and the function $\varphi $ and in appropriate compatibility of $\varphi
$ and $g$ (see e.g. \cite{GM92,JacobIII,MP05}).
\end{remark}

\subsection{Closeness of $u^{\protect\epsilon }(t,x)$ and $u(t,x)$}

\label{ssec:closeness_ue}

We now state and prove the theorem on closeness of $u^{\epsilon }(t,x)$ and $%
u(t,x)$. In what follows we use the same letters $K$ and $C$ for various
positive constants independent of $x,$ $t,$ and $\epsilon .$

\begin{theorem}
\label{Thm_eps} Let Assumptions~\ref{assumption3.1}, \ref{assumption3.2} and %
\ref{assumption2.1} hold, the latter with $l=1$ and $m=3$. Then for $0\leq
\epsilon <1$
\begin{equation}
|u^{\epsilon }(t,x)-u(t,x)|\leq K\int_{|z|\leq \epsilon }|z|^{3}\nu (\mathrm{%
d}z),\ \ (t,x)\in Q,  \label{thme1}
\end{equation}%
where $K>0$ does not depend on $t, x, \epsilon$.
\end{theorem}

\begin{proof}
We have $\big(\tilde{\tau}_{t,x},\tilde{X}_{t,x}(\tilde{\tau}_{t,x})\big)\in
Q^{c}$ and $\varphi \big(\tilde{\tau}_{t,x},\tilde{X}_{t,x}(\tilde{\tau}%
_{t,x})\big)=u\big(\tilde{\tau}_{t,x},\tilde{X}_{t,x}(\tilde{\tau}_{t,x})%
\big),$ and
\begin{equation}
u^{\epsilon }(t,x)-u(t,x)=\mathbb{E}\left[ u\big(\tilde{\tau}_{t,x},\tilde{X}%
_{t,x}(\tilde{\tau}_{t,x})\big)\tilde{Y}_{t,x,1}(\tilde{\tau}_{t,x})+\tilde{Z%
}_{t,x,1,0}(\tilde{\tau}_{t,x})\right] -u(t,x).  \label{thme2}
\end{equation}%
By Ito's formula, we get
\begin{align}
\lefteqn{u(s,\tilde{X}_{t,x}(s))\tilde{Y}_{t,x,1}(s)+\tilde{Z}_{t,x,1,0}(s)}
\label{thme3} \\
& =u(t,x)+\int_{t}^{s}\tilde{Y}_{t,x,1}(s^{\prime })\bigg[\frac{\partial }{%
\partial t}u(s^{\prime },\tilde{X}_{t,x}(s^{\prime }-))+\frac{1}{2}%
\sum_{i,j=1}^{d}a^{ij}(s^{\prime },\tilde{X}_{t,x}(s^{\prime }-))\frac{%
\partial ^{2}u}{\partial x^{i}\partial x^{j}}(s^{\prime },\tilde{X}%
_{t,x}(s^{\prime }-))  \notag \\
& +\langle b(s^{\prime },\tilde{X}_{t,x}(s^{\prime }-)),\nabla u(s^{\prime },%
\tilde{X}_{t,x}(s^{\prime }-))\rangle -\langle F(s,\tilde{X}(s-))\gamma
_{\epsilon },\nabla u(s^{\prime },\tilde{X}_{t,x}(s^{\prime }-))\rangle
\notag \\
& +c(s,\tilde{X}_{t,x}(s^{\prime }-))u(s^{\prime },\tilde{X}_{t,x}(s^{\prime
}-))+g(s^{\prime },\tilde{X}_{t,x}(s^{\prime }-))\bigg]\mathrm{d} s^{\prime}
\notag \\
& +\frac{1}{2}\int_{t}^{s}\tilde{Y}_{t,x,1}(s^{\prime
})\sum_{i,j=1}^{d}\left( F(s^{\prime },\tilde{X}_{t,x}(s^{\prime
}-))B_{\epsilon }F^{\top }(s^{\prime },\tilde{X}_{t,x}(s^{\prime }-))\right)
^{{ij}}\frac{\partial ^{2}u}{\partial x^{i}\partial x^{j}}(s^{\prime },%
\tilde{X}_{t,x}(s^{\prime }-))\mathrm{d} s^{\prime}  \notag \\
& +\int_{t}^{s}\tilde{Y}_{t,x,1}(s^{\prime })\left[ \sigma (s^{\prime },%
\tilde{X}(s^{\prime }-))\nabla u(s^{\prime },\tilde{X}(s^{\prime }-))\right]
^{\top }\mathrm{d}w(s^{\prime })  \notag \\
& +\int_{t}^{s}\tilde{Y}_{t,x,1}(s^{\prime })\left[ F(s^{\prime },\tilde{X}%
(s^{\prime }-))\beta _{\epsilon }\nabla u(s^{\prime },\tilde{X}(s^{\prime
}-))\right] ^{\top }\mathrm{d}W(s^{\prime })  \notag \\
& +\int_{t}^{s}\int_{|z|\geq \epsilon }\tilde{Y}_{t,x,1}(s^{\prime })\left[
u(s^{\prime },\tilde{X}(s-)+F(s^{\prime },\tilde{X}(s^{\prime
}-))z)-u(s^{\prime },\tilde{X}(s^{\prime }))\right] N(\mathrm{d}z,\mathrm{d}%
s^{\prime }).  \notag
\end{align}%
Since $u(t,x)$ solves (\ref{eq:problem}) and recalling (\ref{eq:gen21}), we
obtain from (\ref{thme3}):%
\begin{align}
\lefteqn{u\left( s,\tilde{X}_{t,x}(s)\right) \tilde{Y}_{t,x,1}(s)+\tilde{Z}%
_{t,x,1,0}(s)-u(t,x)}  \label{thme4} \\
& =-\int_{t}^{s}\tilde{Y}_{t,x,1}(s^{\prime })[\langle F(s,\tilde{X}%
(s-))\gamma _{\epsilon },\nabla u(s^{\prime },\tilde{X}_{t,x}(s^{\prime
}-))\rangle  \notag \\
& +\int_{\mathbb{R}^{m}}\{u(s^{\prime },\tilde{X}_{t,x}(s^{\prime
}-)+F(s^{\prime },\tilde{X}_{t,x}(s^{\prime }-))z)-u(s^{\prime },\tilde{X}%
_{t,x}(s^{\prime }-))  \notag \\
& -\langle F(s^{\prime },\tilde{X}_{t,x}(s^{\prime }-))z,\nabla u(s^{\prime
},\tilde{X}_{t,x}(s^{\prime }-))\rangle \mathds{1}_{|z|\leq 1}\}\nu (\mathrm{%
d}z)]\mathrm{d} s^{\prime}  \notag \\
& +\frac{1}{2}\int_{t}^{s}\tilde{Y}_{t,x,1}(s^{\prime
})\sum_{i,j=1}^{d}\left( F(s^{\prime },\tilde{X}_{t,x}(s^{\prime
}-))B_{\epsilon }F^{\top }(s^{\prime },\tilde{X}_{t,x}(s^{\prime }-))\right)
^{ij}\frac{\partial ^{2}u}{\partial x^{i}\partial x^{j}}(s^{\prime },\tilde{X%
}_{t,x}(s^{\prime }-))\mathrm{d} s^{\prime}  \notag \\
& +\int_{t}^{s}\tilde{Y}_{t,x,y}(s^{\prime })\left[ \sigma (s^{\prime },%
\tilde{X}(s^{\prime }-))\nabla u(s^{\prime },\tilde{X}(s^{\prime }-))\right]
^{\top }\mathrm{d}w(s^{\prime })  \notag \\
& +\int_{t}^{s}\tilde{Y}_{t,x,y}(s^{\prime })\left[ F(s^{\prime },\tilde{X}%
(s^{\prime }-))\beta _{\epsilon }\nabla u(s^{\prime },\tilde{X}(s^{\prime
}-))\right] ^{\top }\mathrm{d}W(s^{\prime })  \notag \\
& +\int_{t}^{s}\int_{|z|\geq \epsilon }\tilde{Y}_{t,x,1}(s^{\prime
})[u(s^{\prime },\tilde{X}(s-)+F(s^{\prime },\tilde{X}(s^{\prime
}-))z)-u(s^{\prime },\tilde{X}(s^{\prime }))]N(\mathrm{d}z,\mathrm{d}%
s^{\prime })  \notag \\
& =\frac{1}{2}\int_{t}^{s}\tilde{Y}_{t,x,1}(s^{\prime
})\sum_{i,j=1}^{d}\left( F(s^{\prime },\tilde{X}_{t,x}(s^{\prime
}-))B_{\epsilon }F^{\top }(s^{\prime },\tilde{X}_{t,x}(s^{\prime }-))\right)
^{ij}\frac{\partial ^{2}u}{\partial x^{i}\partial x^{j}}(s^{\prime },\tilde{X%
}_{t,x}(s^{\prime }-))\mathrm{d} s^{\prime}  \notag \\
& -\int_{t}^{s}\int_{|z|<\epsilon }\tilde{Y}_{t,x,1}(s^{\prime
})[u(s^{\prime },\tilde{X}_{t,x}(s^{\prime }-)+F(s^{\prime },\tilde{X}%
_{t,x}(s^{\prime }-))z)-u(s^{\prime },\tilde{X}_{t,x}(s^{\prime }-))  \notag
\\
& -\langle F(s^{\prime },\tilde{X}_{t,x}(s^{\prime }-))z,\nabla u(^{\prime },%
\tilde{X}_{t,x}(s^{\prime }-))\rangle ]\nu (\mathrm{d}z)\mathrm{d} s^{\prime}
\notag \\
& +\int_{t}^{s}\tilde{Y}_{t,x,1}(s^{\prime })\left[ \sigma (s^{\prime },%
\tilde{X}(s^{\prime }-))\nabla u(s^{\prime },\tilde{X}(s^{\prime }-))\right]
^{\top }\mathrm{d}w(s^{\prime })  \notag \\
& +\int_{t}^{s}\tilde{Y}_{t,x,1}(s^{\prime })\left[ F(s^{\prime },\tilde{X}%
(s^{\prime }-))\beta _{\epsilon }\nabla u(s^{\prime },\tilde{X}(s^{\prime
}-))\right] ^{\top }\mathrm{d}W(s^{\prime })  \notag \\
& +\int_{t}^{s}\int_{|z|\geq \epsilon }\tilde{Y}_{t,x,1}(s^{\prime })\left[
u(s^{\prime },\tilde{X}(s-)+F(s^{\prime },\tilde{X}(s^{\prime
}-))z)-u(s^{\prime },\tilde{X}(s^{\prime }))\right]  \notag \\
& \qquad \times (N(\mathrm{d}z,\mathrm{d}s^{\prime })-\nu (\mathrm{d}z)%
\mathrm{d}s^{\prime }).
\end{align}

Replacing $s$ with the stopping time $\tilde{\tau}_{t,x}$ in (\ref{thme4})
(cf. (\ref{thme2})), taking expectations of the resulting left- and
right-hand sides of (\ref{thme4}) and using the martingale property, we
arrive at
\begin{align}
\lefteqn{\mathbb{E}\left[ u\left( \tilde{\tau}_{t,x},\tilde{X}_{t,x}(\tilde{%
\tau}_{t,x})\right) \tilde{Y}_{t,x,1}(\tilde{\tau}_{t,x})+\tilde{Z}%
_{t,x,1,0}(\tilde{\tau}_{t,x})\right] -u(t,x)}  \label{thme5} \\
& =\mathbb{E}\int_{t}^{\tilde{\tau}_{t,x}}\tilde{Y}_{t,x,1}(s) \bigg[ \frac{1%
}{2}\sum_{i,j=1}^{d} \left( F(s,\tilde{X}_{t,x}(s-))B_{\epsilon }F^{\top }(s,%
\tilde{X}_{t,x}(s-))\right) ^{ij} \frac{\partial ^{2}u}{\partial
x^{i}\partial x^{j}}(s,\tilde{X}_{t,x}(s-))  \notag \\
& -\int_{|z|<\epsilon } \Big(u(s,\tilde{X}_{t,x}(s-)+F(s,\tilde{X}%
_{t,x}(s-))z)-u(s,\tilde{X}_{t,x}(s-))  \notag \\
& -\langle F(s,\tilde{X}_{t,x}(s-))z,\nabla u(s,\tilde{X}_{t,x}(s-))\rangle %
\Big) \nu (\mathrm{d}z)\bigg]\mathrm{d} s.  \notag
\end{align}%
By Taylor's expansion, we get
\begin{align}
\lefteqn{u(s,\tilde{X}_{t,x}(s-)+F(s,\tilde{X}_{t,x}(s-))z)-u(s,\tilde{X}%
_{t,x}(s-))}  \label{thme6} \\
& =\langle F(s,\tilde{X}_{t,x}(s-))z,\nabla u(s,\tilde{X}_{t,x}(s-))\rangle
\notag \\
& +\frac{1}{2}\sum_{i,j=1}^{d}\left( F(s,\tilde{X}_{t,x}(s-))z)^{i}(F(s,%
\tilde{X}_{t,x}(s-))z\right) ^{j}\frac{\partial ^{2}u}{\partial
x^{i}\partial x^{j}}(s,\tilde{X}_{t,x}(s-))  \notag \\
& +\frac{1}{6}\sum_{i,j,k=1}^{d}\left( F(s,\tilde{X}_{t,x}(s-))z)^{i}(F(s,%
\tilde{X}_{t,x}(s-))z\right) ^{j}(F(s,\tilde{X}_{t,x}(s-))z)^{k}  \notag \\
& \times \frac{\partial ^{3}u}{\partial x^{i}\partial x^{j}\partial x^{k}}(s,%
\tilde{X}_{t,x}(s-)+\theta F(s,\tilde{X}_{t,x}(s-))z),  \notag
\end{align}%
where $\theta \in \lbrack 0,1].$ Recalling (\ref{eq:gen22}), we obtain from (%
\ref{thme5})-(\ref{thme6}):
\begin{align}
\lefteqn{\mathbb{E}\left[ u\left( \tilde{\tau}_{t,x},\tilde{X}_{t,x}(\tilde{%
\tau}_{t,x})\right) \tilde{Y}_{t,x,1}(\tilde{\tau}_{t,x})+\tilde{Z}%
_{t,x,1,0}(\tilde{\tau}_{t,x})\right] -u(t,x)}  \label{thme7} \\
& =-\frac{1}{6}\mathbb{E}\int_{t}^{\tilde{\tau}_{t,x}}\tilde{Y}%
_{t,x,1}(s)\int_{|z|<\epsilon }\sum_{i,j,k=1}^{d}\left( F(s,\tilde{X}%
_{t,x}(s-))z)^{i}(F(s,\tilde{X}_{t,x}(s-))z\right) ^{j}  \notag \\
& \times (F(s,\tilde{X}_{t,x}(s-))z)^{k}\frac{\partial ^{3}u}{\partial
x^{i}\partial x^{j}\partial x^{k}}(s,\tilde{X}_{t,x}(s-)+\theta F(s,\tilde{X}%
_{t,x}(s-))z)\nu (\mathrm{d}z)\mathrm{d} s.  \notag
\end{align}%
By definition of $\tilde{\tau}_{t,x}$, $\tilde{X}_{t,x}(s-)\in G$ for $s\leq
\tilde{\tau}_{t,x},$ then we have
\begin{eqnarray}
\left\vert F(s,\tilde{X}_{t,x}(s-))\right\vert &\leq &\max_{t_{0}\leq s\leq
T,\ x\in \bar{G}}|F(s,x)|\leq K,  \label{thme8} \\
\left\vert \tilde{X}_{t,x}(s-)+\theta F(s,\tilde{X}_{t,x}(s-))z\right\vert
&\leq &\max_{x\in \bar{G}}|x|+\epsilon \max_{t_{0}\leq s\leq T,\ x\in \bar{G}%
}|F(s,x)|\leq K,  \notag
\end{eqnarray}%
where $K>0$ does not depend on $\epsilon, t, x, s$, noting that $%
|z|<\epsilon $. Using Assumption~\ref{assumption2.1}, (\ref{thme7})-(\ref%
{thme8})
\begin{eqnarray}
&&\left\vert \mathbb{E}\left[ u\left( \tilde{\tau}_{t,x},\tilde{X}_{t,x}(%
\tilde{\tau}_{t,x})\right) \tilde{Y}_{t,x,1}(\tilde{\tau}_{t,x})+\tilde{Z}%
_{t,x,1,0}(\tilde{\tau}_{t,x})\right] -u(s,x)\right\vert  \label{thme9} \\
&\leq &K\int_{t_{0}}^{T}\mathbb{E}\tilde{Y}_{t,x,1}(s)\chi _{\tilde{\tau}%
_{t,x}>s}\mathrm{d} s\cdot \int_{|z|<\epsilon }|z|^{3}\nu (\mathrm{d}z).
\notag
\end{eqnarray}%
Since $c\big(s,\tilde{X}_{t,x}(s)\big)$ is bounded on the set $\{\tilde{\tau}%
_{t,x}>s\},$ $\mathbb{E}\tilde{Y}_{t,x,1}(s)\chi _{\tilde{\tau}_{t,x}>s}$ is
bounded which together with (\ref{thme9}) implies (\ref{thme1}).
\end{proof}

\begin{example}[Tempered $\protect\alpha $-stable Process]
\label{ex:alpha-stable} For $\alpha \in (0,2)$ consider an $\alpha $-stable
process with L\'evy measure given by $\nu (\mathrm{d}z)=z^{-1-\alpha }%
\mathrm{d}z$. Then
\begin{equation*}
\int_{|z|\leq \epsilon }|z|^{3}\nu (\mathrm{d}y)=\frac{\epsilon ^{3-\alpha }%
}{3-\alpha }.
\end{equation*}%
Similarly, for a tempered stable distribution which has L\'evy measure given
by
\begin{equation*}
\nu (\mathrm{d}z)=\Big(\frac{C_{+}\mathrm{e}^{-\lambda _{+}z}}{z^{1+\alpha }}%
\mathbf{1}_{z>0}+\frac{C_{-}\mathrm{e}^{-\lambda _{-}z}}{z^{1+\alpha }}%
\mathbf{1}_{z<0}\Big)\mathrm{d}z,
\end{equation*}%
for $\alpha \in (0,2)$ and $C_{+},$ $C_{-},$ $\lambda _{+},$ $\lambda _{-}>0$
we find that the error from approximating the small jumps by diffusion as in
Theorem~\ref{Thm_eps}\ is of the order $O(\epsilon ^{3-\alpha })$.
\end{example}

\section{Weak approximation of jump-diffusions in bounded domains}

\label{sec:weak}

In this section we propose and study a numerical algorithm which weakly
approximates the solutions of the jump-diffusion (\ref{eq:gensde2}), (\ref%
{eq:fkac2y})-(\ref{eq:fkac2z}) with finite intensity of jumps in a bounded
domain, i.e., approximates $u^{\epsilon }(t,x)$ from (\ref{eq:fkac2}). In
Section~\ref{sec:alg} we formulate the algorithm based on a simplest random
walk. We analyse the one-step error of the algorithm in Section~\ref{sec:one}
and the global error in Section~\ref{sec:glob}. In Section~\ref{sec:Cau} we
comment on how the global error can be estimated in the Cauchy case. In
Section~\ref{sec:iij} we combine the convergence result of Section~\ref%
{sec:glob} with Theorem~\ref{Thm_eps} to get error estimates in the case of
infinite activity of jumps.

\subsection{Algorithm\label{sec:alg}}

Let us describe an algorithm for simulating a Markov chain that approximates
a trajectory of (\ref{eq:gensde2}), (\ref{eq:fkac2y})-(\ref{eq:fkac2z}). In
what follows we assume that we can exactly sample increments $\delta $
between jump times with the intensity
\begin{equation}
\lambda _{\epsilon }:=\int_{|z|>\epsilon }\nu (\mathrm{d}z)  \label{eq:lame}
\end{equation}%
and jump sizes $J_{\epsilon }$ are distributed according to the density
\begin{equation}
\rho _{\epsilon }(z):=\frac{\nu (z)\mathbf{I}_{|z|>\epsilon }}{\lambda
_{\epsilon }}.  \label{eq:rhoe}
\end{equation}

\begin{remark}
There are known methods for simulating jump times and sizes for many
standard distributions. In general, if there exists an explicit expression
for the jump size density, one can construct a rejection method to sample
jump sizes. An overview with regard to simulation of jump times and sizes
can be found in~\cite{ContTankov,devroye:1986}.
\end{remark}

In what follows we also require the following to hold.

\begin{assumption}[Moments of $J$]
\label{assum:moments-jumps} \textit{There exists a constant} $K>0$ \textit{%
independent of} $\epsilon $ \textit{such that}
\begin{equation*}
\mathbb{E}|J_{\epsilon }|^{p}\equiv \frac{1}{\lambda _{\epsilon }}%
\int_{|z|>\epsilon }|z|^{p}\nu (dz)\leq K
\end{equation*}%
\textit{for sufficiently large} $p\geq 2.$
\end{assumption}

We also note that
\begin{equation}
\frac{\gamma _{\epsilon }^{2}}{\lambda _{\epsilon }}\leq K,  \label{eq:gala}
\end{equation}%
where $K>0$ is a constant independent of $\varepsilon .$

We now describe the algorithm. Fix a time-discretization step $h>0$ and
suppose the current position of the chain is $(t,x,y,z)$. If the jump time
increment $\delta <h$, we set $\theta =\delta $, otherwise $\theta =h$, i.e.
$\theta =\delta \wedge h$.

In the case $\theta =h$, we apply the weak explicit Euler approximation with
the simplest simulation of noise to the system (\ref{eq:gensde2}), (\ref%
{eq:fkac2y})-(\ref{eq:fkac2z}) with no jumps:
\begin{eqnarray}
\tilde{X}_{t,x}(t+\theta ) &\approx &X=x+\theta \cdot \left(
b(t,x)-F(t,x)\gamma _{\epsilon }\right) +\sqrt{\theta }\cdot \left( \sigma
(t,x)\,\xi \,+F(t,x)\beta _{\epsilon }\ \eta \right) ,\ \ \ \ \
\label{Hc01} \\
\tilde{Y}_{t,x,y}(t+\theta ) &\approx &Y=y+\theta \cdot c(t,x)\,y\,,
\label{Hc02} \\
\tilde{Z}_{t,x,y,z}(t+\theta ) &\approx &Z=z+\theta \cdot g(t,x)\,y\,,
\label{Hc03}
\end{eqnarray}%
where $\xi =(\xi ^{1},\ldots ,\xi ^{d})^{\intercal }$, $\eta =(\eta
^{1},\ldots ,\eta ^{m})^{\intercal }$, with $\xi ^{1}, \dots, \xi^{d}$ and $%
\eta ^{1}, \dots, \eta^{m}$ mutually independent random variables, taking
the values $\pm 1$ with equal probability. In the case of $\theta <h$, we
replace (\ref{Hc01}) by the following explicit Euler approximation
\begin{eqnarray}
\tilde{X}_{t,x}(t+\theta ) &\approx &X=x+\theta \cdot \left(
b(t,x)-F(t,x)\gamma _{\epsilon }\right) +\sqrt{\theta }\cdot \left( \sigma
(t,x)\,\xi \,+F(t,x)\beta _{\epsilon }\ \eta \right)  \label{Hc04} \\
&&+F(t,x)J_{\epsilon }.  \notag
\end{eqnarray}

Let $(t_{0},x_{0})\in Q$. We aim to find the value $u^{\epsilon
}(t_{0},x_{0})$, where $u^{\epsilon }(t,x)$\ solves the problem (\ref{v1}).
Introduce a discretization of the interval $\left[ t_{0},T\right] $, for
example the equidistant one:
\begin{equation*}
h:=(T-t_{0})/L.
\end{equation*}

To approximate the solution of the system (\ref{eq:gensde2}), we construct a
Markov chain $(\vartheta _{k},X_{k},Y_{k},Z_{k})$ which stops at a random
step $\varkappa $ when $(\vartheta _{k},X_{k})$ exits the domain $Q.$ The
algorithm is formulated as Algorithm~\ref{Hca01} below.

\begin{algorithm}[htb]
\caption{Algorithm for \eqref{eq:gensde2},
\eqref{eq:fkac2y}-\eqref{eq:fkac2z}.}\label{Hca01}  \hspace*{%
\algorithmicindent}\textbf{Output:} $\bar{\vartheta}_{\varkappa},
X_\varkappa, Y_\varkappa, Z_\varkappa$

\begin{algorithmic}[1]
\State\noindent \textbf{Initialize:} $\vartheta
_{0}=t_{0},\;X_{0}=x_{0},\;Y_{0}=1,\;Z_{0}=0,\;k=0.$ \State\textbf{Simulate:}
$\xi _{k}$ and $\eta _{k}$ with i.i.d.\ components taking values $\pm 1$ with
probability $1/2$ and independently $I_{k}\sim \mathrm{Bernoulli}\big(%
1-e^{-\lambda _{\epsilon }h}\big)$.\label{alg:step1} \If{$I_{k}=0,$} \State%
\textbf{Set:} $\theta _{k}=h$ \State\textbf{Evaluate:} $X_{k+1}$, $Y_{k+1}$,
$Z_{k+1}$ according to $(\ref{Hc01})-(\ref{Hc03})$ with $t=\vartheta _{k}$, $%
\theta =\theta _{k},$ $\xi =\xi _{k},$ $\eta =\eta _{k},$ $x=X_{k}$, $y=Y_{k}
$, $z=Z_{k}$. \Else \State\textbf{Sample:} $\delta _{k}$ according to the
density $\dfrac{\lambda _{\epsilon }e^{-\lambda _{\epsilon }x}}{%
1-e^{-\lambda _{\epsilon }h}}$ with finite support $[0,h].$ \State \textbf{Set:} $%
\theta _{k}=\delta _{k}$ \State\textbf{Sample:} jump size $J_{\epsilon ,k}$
according to the density $(\ref{eq:rhoe})$. \State\textbf{Evaluate:} $X_{k+1}
$, $Y_{k+1}$ and $Z_{k+1}$ according to $(\ref{Hc04})$, $(\ref{Hc02})$, $(\ref{Hc03})$
 with $t=\vartheta _{k}$, $\theta =\theta _{k},$ $\xi =\xi
_{k},$ $\eta =\eta _{k},$ $J_{\epsilon }=J_{\epsilon ,k}$, $x=X_{k}$, $%
y=Y_{k}$, $z=Z_{k}$.\EndIf \State\textbf{Set:} $\vartheta _{k+1}=\vartheta
_{k}+\theta _{k}.$ \If{$\vartheta _{k+1}\geq T$ or $X_{k+1}\notin G$} \State%
\textbf{Set:} $X_{\varkappa }=X_{k+1},$ $Y_{\varkappa }=X_{k+1},$ $%
Z_{\varkappa }=Z_{k+1},$ $\varkappa =k+1$ \If{$\vartheta _{k+1}<T$}%
\thinspace\ \textbf{Set:} $\bar{\vartheta}_{\varkappa }=\vartheta _{k+1}$ %
\Else\thinspace\ \textbf{Set:} $\bar{\vartheta}_{\varkappa }=T$ \EndIf \State%
\textbf{STOP} \Else\State\textbf{Set:} $k=k+1$ and GOTO \ref{alg:step1}. %
\EndIf
\end{algorithmic}
\end{algorithm}

\begin{remark}
We note \cite{MT03,MT04} that in the diffusion case (i.e., when there is no
jump component in the noise which drives SDEs) solving Dirichlet problems
for parabolic or elliptic PDEs requires to complement a random walk inside
the domain $G$ with a special approximation near the boundary $\partial G$.
In contrast, in the case of Dirichlet problems for PIDEs we do not need a
special construction near the boundary since the boundary condition is
defined on the whole complement $G^{c}.$ Here, when the chain $X_{k}$ exits $%
G,$ we know the exact value of the solution $u(\bar{\vartheta}_{\varkappa
},X_{\varkappa })=\varphi (\bar{\vartheta}_{\varkappa },X_{\varkappa })$ at
the exit point $(\bar{\vartheta}_{\varkappa },X_{\varkappa })$, while in the
diffusion case when a chain exits $G,$ we do not know the exact value of the
solution at the exit point and need an approximation. Due to this fact,
Algorithm~\ref{Hca01} is somewhat simpler than algorithms for Dirichlet
problems for parabolic or elliptic PDEs (cf. \cite{MT03,MT04} and references
therein).
\end{remark}

\subsection{One-step error\label{sec:one}}

In this section we consider the one-step error of Algorithm~\ref{Hca01}. The
one step of this algorithm takes the form for $(t,x)\in Q:$
\begin{align}
X& =x+\theta \left( b(t,x)-F(t,x)\gamma _{\epsilon }\right) +\sqrt{\theta }%
\left( \sigma (t,x)\xi +F(t,x)\beta _{\epsilon }\eta \right) +\mathbf{I}%
(\delta <h)F(t,x)J_{\epsilon },\ \ \   \label{v_one} \\
Y& =y+\theta c(t,x)y,  \label{v_one2} \\
Z& =z+\theta g(t,x)y.  \label{v_one3}
\end{align}

Before we state and prove an error estimate for the one-step of Algorithm~%
\ref{Hca01}, we need to introduce some additional notation. For brevity let
us write $b=b(t,x)$, $\sigma =\sigma (t,x)$, $F=F(t,x)$, $g=g(t,x)$, $%
c=c(t,x)$, $J=J_{\epsilon }$. Let us define the intermediate points $Q_{i}$
and their differences $\Delta _{i}$, for $i=1,\ldots ,4$ :
\begin{align}
\Delta _{1}& =\theta ^{1/2}\left[ \sigma \xi +F\beta _{\epsilon }\eta \right]
,  \label{dQ} \\
\Delta _{2}& =\theta \left[ b-F\gamma _{\epsilon }\right] ,  \notag \\
\Delta _{3}& =\mathbf{I}(\delta <h)FJ,  \notag \\
Q_{1}& =x+\Delta _{1}+\Delta _{2}+\Delta _{3}=X,  \notag \\
Q_{2}& =x+\Delta _{2}+\Delta _{3},  \notag \\
Q_{3}& =x+\Delta _{3},  \notag \\
Q_{4}& =x,  \notag
\end{align}%
where $x\in G.$ Note that $Q_{i}$, $i=1,\ldots ,3,$ can be outside $G.$

\begin{lemma}[Moments of intermediate points $Q_{i}$]
\label{lem:growth-intermediate-points} Under Assumptions~\ref{assumption3.1}
and~\ref{assum:moments-jumps}, there is $K>0$ independent of $\epsilon $ and
$h$ such that for $p\geq 1$:
\begin{align}
\mathbb{E}\left[ |Q_{i}|^{2p}\big|\theta ,t,x\right] & \leq K(1+\theta
^{2p}\gamma _{\epsilon }^{2p}),\ i=1,2,  \label{newbQ1} \\
\mathbb{E}\left[ |Q_{i}|^{2p}\big|\theta ,t,x\right] & \leq K,\ i=3,4,
\label{bQ2}
\end{align}%
where $Q_{i}$ are defined in (\ref{dQ}).
\end{lemma}

\begin{proof}
It is not difficult to see that the points $Q_{i},$ $i=1,2,$ are of the
following form
\begin{equation*}
Q_{i}=x+c_{1}\theta ^{1/2}\left[ \sigma \xi +F(t,x)\beta _{\epsilon }\eta %
\right] +\theta \left[ b(t,x)-F(t,x)\gamma _{\epsilon }\right] +\mathbf{I}%
(\theta <h)F(t,x)J_{\epsilon },
\end{equation*}%
where $c_{1}$ is either $0$ or $1$. It is obvious that $\xi $ and $\eta $
and their moments are all bounded. The functions $b(t,x),$ $\sigma (t,x)$
and $F(t,x)$ are bounded as $(t,x)\in Q$, and for $x\in G$, $|x|^{2p}$ is
also bounded. Recall that sufficiently high moments of $J_{\epsilon }$ are
bounded due to Assumption~\ref{assum:moments-jumps}. \ Then, using the
Cauchy-Schwarz inequality, we can show that
\begin{align*}
\mathbb{E}\left[ |Q_{i}|^{2p}\big|\theta ,t,x\right] & \leq |x|^{2p}+K\theta
^{p}+K\theta ^{2p}\left[ 1+\gamma _{\epsilon }^{2p}\right] +K\mathbf{I}%
(\theta <h)\mathbb{E}|J_{\epsilon }|^{2p} \\
& \leq K(1+\theta ^{2p}\gamma _{\epsilon }^{2p}).
\end{align*}%
Hence, we obtained \eqref{newbQ1}. The bound (\ref{bQ2}) is shown
analogously.
\end{proof}

It is not difficult to prove the following technical lemma.

\begin{lemma}[Moments of $\protect\theta $]
\label{lem:momth}For integer $p\geq 2,$ we have%
\begin{equation*}
\mathbb{E}\theta ^{p}\leq K\frac{1-e^{-\lambda _{\epsilon }h}(1+\lambda
_{\epsilon }h)}{\lambda _{\epsilon }^{p}},
\end{equation*}%
where $K>0$ depends on $p$ but is independent of $\lambda _{\epsilon }$ and $%
h.$
\end{lemma}

Now we prove an estimate for the one-step error.

\begin{theorem}[One--step error of Algorithm~\protect\ref{Hca01}]
\label{Thm:Alg1onestep} Under Assumption~\ref{assumption2.1'} with $l=2,m=4$
and Assumptions~\ref{assumption3.1}, \ref{assum:differentiability} and \ref%
{assum:moments-jumps} the one--step error of Algorithm~\ref{Hca01} given by
\begin{equation*}
R(t,x,y,z):=u^{\epsilon }(t+\theta ,X)Y+Z-u^{\epsilon }(t,x)y-z
\end{equation*}%
satisfies the bound
\begin{equation}
\big|\mathbb{E}[R(t,x,y,z)]\big|\leq K(1+\gamma _{\epsilon }^{2})\frac{%
1-e^{-\lambda _{\epsilon }h}(1+\lambda _{\epsilon }h)}{\lambda _{\epsilon
}^{2}}y,  \label{one-step}
\end{equation}%
where $K>0$ is a constant independent of $h$ and $\epsilon $.
\end{theorem}

\begin{proof}
For any smooth function $v(t,x)$, we write $D_{l}v_{n}=(D_{l}v)(t,Q_{n})$
for the $l$-th time derivative and $(D_{l}^{k}v)(t,x)[f_{1},\ldots ,f_{k}]$
for the $l$-th time derivative of the $k$-th spatial directional derivative
evaluated in the direction $[f_{1},\ldots ,f_{k}]$. For example, if $k=2$
and $l=1$,
\begin{equation*}
D_{1}^{2}v[f_{1},f_{2}]=\sum\limits_{i=1}^{d}\sum%
\limits_{j=1}^{d}f_{1,i}f_{2,j}\frac{\partial ^{3}v}{\partial tx_{i}x_{j}}.
\end{equation*}%
We will also use the following short notation%
\begin{equation*}
D_{l}^{k}v_{i}[f_{1},\ldots ,f_{k}]:=(D_{l}^{k}v)(t,Q_{i})[f_{1},\ldots
,f_{k}].
\end{equation*}

The final aim of this theorem is to achieve an error estimate explicitly
capturing the (singular) dependence of the one-step error on $\epsilon$. To
this end, we split the error into several parts according to the
intermediate points $Q_{i}$ defined in (\ref{dQ}).

Using (\ref{v_one}) and (\ref{dQ}), we have
\begin{align*}
u^{\epsilon }(t+\theta ,X)& =u^{\epsilon }(t+\theta ,Q_{1}) \\
& =u^{\epsilon }\Big(t+\theta ,x+\mathbf{I}(\delta <h)FJ+\theta (b-F\gamma
_{\epsilon })+\theta ^{1/2}(\sigma \xi +F\beta _{\epsilon }\eta )\Big) \\
& =u^{\epsilon }\Big(t+\theta ,x+\Delta _{1}+\Delta _{2}+\Delta _{3}\Big).
\end{align*}%
To precisely account for the factor $\gamma _{\epsilon }$ and powers of $%
\theta $ in the analysis of the one-step error, we use multiple Taylor
expansions of $u^{\epsilon }(t+\theta ,X).$ We obtain
\begin{align}
u^{\epsilon }(t+\theta ,X)& =u^{\epsilon }(t,Q_{1})+\theta
D_{1}u_{1}^{\epsilon }+R_{11}  \label{ost1} \\
& =u^{\epsilon }(t,Q_{2})+D^{1}u_{2}^{\epsilon }[\Delta _{1}]+\frac{1}{2}%
D^{2}u_{2}^{\epsilon }[\Delta _{1},\Delta _{1}]  \notag \\
& +\frac{1}{6}D^{3}u_{2}^{\epsilon }[\Delta _{1},\Delta _{1},\Delta
_{1}]+\theta D_{1}u_{2}^{\epsilon }+\theta D_{1}^{1}u_{2}^{\epsilon }[\Delta
_{1}]  \notag \\
& \quad +R_{11}+R_{12}+R_{13}  \notag \\
& =u^{\epsilon }(t,Q_{3})+D^{1}u_{3}^{\epsilon }[\Delta
_{2}]+D^{1}u_{2}^{\epsilon }[\Delta _{1}]+\frac{1}{2}D^{2}u_{3}^{\epsilon
}[\Delta _{1},\Delta _{1}]+\frac{1}{6}D^{3}u_{2}^{\epsilon }[\Delta
_{1},\Delta _{1},\Delta _{1}]  \notag \\
& \quad +\theta D_{1}u_{3}^{\epsilon }+\theta D_{1}^{1}u_{2}^{\epsilon
}[\Delta _{1}]+R_{11}+R_{12}+R_{13}+R_{14}+R_{15}+R_{16}  \notag \\
& =u^{\epsilon }(t,Q_{3})+D^{1}u_{4}^{\epsilon }[\Delta
_{2}]+D^{1}u_{2}^{\epsilon }[\Delta _{1}]+\frac{1}{2}D^{2}u_{4}^{\epsilon
}[\Delta _{1},\Delta _{1}]  \notag \\
& \quad +\frac{1}{6}D^{3}u_{2}^{\epsilon }[\Delta _{1},\Delta _{1},\Delta
_{1}]+\theta D_{1}u_{4}^{\epsilon }+\theta D_{1}^{1}u_{2}^{\epsilon }[\Delta
_{1}]+R_{1},  \notag
\end{align}%
where the remainders are as follows
\begin{align*}
R_{11}& =\frac{1}{2}\theta ^{2}\int_{0}^{1}sD_{2}u^{\epsilon }\Big(%
t+(1-s)\theta ,Q_{1}\Big)ds, \\
R_{12}& =\frac{1}{24}\int_{0}^{1}s^{3}D^{4}u^{\epsilon
}(t,sQ_{2}+(1-s)Q_{1})[\Delta _{1},\Delta _{1},\Delta _{1},\Delta _{1}]ds, \\
R_{13}& =\frac{1}{2}\theta \int_{0}^{1}s^{2}D_{1}^{2}u^{\epsilon
}(t,sQ_{2}+(1-s)Q_{1})[\Delta _{1},\Delta _{1}]ds, \\
R_{14}& =\frac{1}{2}\int_{0}^{1}sD^{2}u^{\epsilon
}(t,s(Q_{3}+(1-s)Q_{2})[\Delta _{2},\Delta _{2}]ds, \\
R_{15}& =\frac{1}{2}\int_{0}^{1}s^{2}D^{3}u^{\epsilon
}(t,s(Q_{3})+(1-s)Q_{2})[\Delta _{1},\Delta _{1},\Delta _{2}]ds, \\
R_{16}& =\theta \int_{0}^{1}sD_{1}^{1}u^{\epsilon
}(t,s(Q_{3})+(1-s)Q_{2})[\Delta _{2}]ds, \\
R_{17}& =\int_{0}^{1}sD^{2}u^{\epsilon }(t,s(Q_{4})+(1-s)Q_{3})[\Delta
_{2},\Delta _{3}]ds, \\
R_{18}& =\frac{1}{2}\int_{0}^{1}sD^{3}u^{\epsilon
}(t,s(Q_{4})+(1-s)Q_{3})[\Delta _{1},\Delta _{1},\Delta _{3}]ds, \\
R_{19}& =\theta \int_{0}^{1}sD_{1}^{1}u^{\epsilon
}(t,s(Q_{4})+(1-s)Q_{3})[\Delta _{3}]ds, \\
R_{1}& =R_{11}+R_{12}+R_{13}+R_{14}+R_{15}+R_{16}+R_{17}+R_{18}+R_{19}.
\end{align*}%
Using (\ref{ost1}), (\ref{v_one2})-(\ref{v_one3}), and the fact that $\xi $
and $\eta $ have mean zero and that components of $\xi ,$ $\eta ,$ $\theta ,$
$J$ are mutually independent, we obtain
\begin{align}
& \mathbb{E}[u^{\epsilon }(t+\theta ,X)Y+Z]  \label{ost2} \\
& =\mathbb{E}\Big[\left( u^{\epsilon }(t,Q_{3})+D^{1}u_{4}^{\epsilon
}[\Delta _{2}]+\frac{1}{2}D^{2}u_{4}^{\epsilon }[\Delta _{1},\Delta
_{1}]+\theta D_{1}u_{4}^{\epsilon }\right) (y+\theta cy)  \notag \\
& \quad +z+\theta gy+y(1+\theta c)R_{1}\Big].  \notag
\end{align}

The following elementary formulas are needed for future calculations:
\begin{gather}
\mathbb{E}\left( D^{2}u^{\epsilon }[\Delta _{1},\Delta _{1}]|\theta \right)
\label{ost3} \\
=\theta \sum_{i,j=1}^{d}\left[ a^{ij}(t,x)+\left( F(t,x)B_{\epsilon
}(t,x)F^{\top }(t,x)\right) ^{ij}\right] \frac{\partial ^{2}u^{\epsilon }}{%
\partial x^{i}\partial x^{j}}  \notag \\
=:\theta (a+FB_{\epsilon }F^{T}):\nabla \nabla u^{\epsilon },  \notag \\
u^{\epsilon }(t,Q_{3})-u^{\epsilon }(t,x)=u^{\epsilon }(t,x+\mathbf{I}%
(\theta <h)FJ)-u^{\epsilon }(t,x)  \notag \\
=\mathbf{I}(\theta <h)[u^{\epsilon }(t,x+FJ)-u^{\epsilon }(t,x)],  \notag \\
\mathbb{E}[\theta ]=\frac{1-e^{-\lambda _{\epsilon }h}}{\lambda _{\epsilon }}%
,  \notag \\
\mathbb{E}[\theta ^{2}]=2\frac{1-e^{-\lambda _{\epsilon }h}(1+\lambda
_{\epsilon }h)}{\lambda _{\epsilon }^{2}},  \notag \\
\mathbb{E}[\mathbf{I}(\theta <h)]=1-e^{-\lambda _{\epsilon }h},  \notag \\
\mathbb{E}[\mathbf{I}(\theta <h)\theta ]=\frac{1-e^{-\lambda _{\epsilon
}h}(1+\lambda _{\epsilon }h)}{\lambda _{\epsilon }}.  \notag
\end{gather}%
Also, $\mathbb{E}v(J)$ for some $v(z)$ will mean%
\begin{equation*}
\mathbb{E}[v(J)]=\mathbb{E}[v(J_{\epsilon })]=\frac{1}{\lambda _{\epsilon }}%
\int_{|s|>\epsilon }v(s)\nu (ds).
\end{equation*}%
Noting that $u_{4}^{\epsilon }=u^{\epsilon }(t,x)=u^{\epsilon }$ and using (%
\ref{ost2}), (\ref{dQ}), (\ref{ost3}) and (\ref{v1}), we obtain
\begin{align*}
& \mathbb{E}R:=\mathbb{E}\big[u^{\epsilon }(t+\theta ,X)Y+Z-u^{\epsilon }y-z%
\big] \\
& =\mathbb{E}[\theta \Big(D_{1}u^{\epsilon }+D^{1}u^{\epsilon }[b-F\gamma
_{\epsilon }]+\frac{1}{2}(a+FB_{\epsilon }F^{T}):\nabla \nabla u^{\epsilon }%
\Big)(y+\theta cy)+\theta gy \\
& +u^{\epsilon }(t,x+\mathbf{I}(\theta <h)FJ)(y+\theta cy)-u^{\epsilon }y%
\big]+y\mathbb{E}[(1+\theta c)R_{1}] \\
& =\mathbb{E}[\theta \Big(D_{1}u^{\epsilon }+D^{1}u^{\epsilon }[b-F\gamma
_{\epsilon }]+\frac{1}{2}(a+FB_{\epsilon }F^{T}):\nabla \nabla u^{\epsilon
}+cu^{\epsilon }+g\Big)y \\
& +[u^{\epsilon }(t,x+\mathbf{I}(\theta <h)FJ)-u^{\epsilon })]y \\
& +\theta ^{2}\Big(D_{1}u^{\epsilon }+D^{1}u^{\epsilon }[b-F\gamma
_{\epsilon }]+\frac{1}{2}(a+FB_{\epsilon }F^{T}):\nabla \nabla u^{\epsilon }%
\Big)cy \\
& +\theta \left[ u^{\epsilon }(t,x+\mathbf{I}(\theta <h)FJ)-u^{\epsilon }%
\right] cy\Big]+y\mathbb{E}[(1+\theta c)R_{1}] \\
& =\mathbb{E}[\theta \Big(D_{1}u^{\epsilon }+D^{1}u^{\epsilon }[b-F\gamma
_{\epsilon }]+\frac{1}{2}(a+FB_{\epsilon }F^{T}):\nabla \nabla u^{\epsilon
}+cu^{\epsilon }+g\Big)y \\
& +\mathbf{I}(\theta <h)[u^{\epsilon }(t,x+FJ)-u^{\epsilon })]y \\
& +\theta ^{2}\Big(D_{1}u^{\epsilon }+D^{1}u^{\epsilon }[b-F\gamma
_{\epsilon }]+\frac{1}{2}(a+FB_{\epsilon }F^{T}):\nabla \nabla u^{\epsilon }%
\Big)cy \\
& +\theta \mathbf{I}(\theta <h)[u^{\epsilon }(t,x+FJ)-u^{\epsilon }]cy]+y%
\mathbb{E}[(1+\theta c)R_{1}] \\
& =\mathbb{E}[\theta ]\Big(D_{1}u^{\epsilon }+D^{1}u^{\epsilon }[b-F\gamma
_{\epsilon }]+\frac{1}{2}(a+FB_{\epsilon }F^{T}):\nabla \nabla u^{\epsilon
}+cu^{\epsilon }+g\Big)y \\
& +\mathbb{E}\left[ \mathbf{I}(\theta <h)[u^{\epsilon }(t,x+FJ)-u^{\epsilon
}(t,x)]y\right] +y\mathbb{E}[R_{1}(1+\theta c)+R_{2}] \\
& =\frac{1-e^{-\lambda _{\epsilon }h}}{\lambda _{\epsilon }}\Big(%
D_{1}u^{\epsilon }+D^{1}u^{\epsilon }[b-F\gamma _{\epsilon }]+\frac{1}{2}%
(a+FB_{\epsilon }F^{T}):\nabla \nabla u^{\epsilon }+cu^{\epsilon }(t,x)+g%
\Big)y \\
& +\left( 1-e^{-\lambda _{\epsilon }h}\right) \mathbb{E}\left[ u^{\epsilon
}(t,x+FJ)-u^{\epsilon }(t,x)\right] y+y\mathbb{E}[R_{0}] \\
& =\frac{1-e^{-\lambda _{\epsilon }h}}{\lambda _{\epsilon }}\Big(%
D_{1}u^{\epsilon }+D^{1}u^{\epsilon }[b-F\gamma _{\epsilon }]+\frac{1}{2}%
(a+FB_{\epsilon }F^{T}):\nabla \nabla u^{\epsilon }+cu^{\epsilon }(t,x)+g%
\Big)y \\
& +\frac{1-e^{-\lambda _{\epsilon }h}}{\lambda _{\epsilon }}%
\int\limits_{|s|>\epsilon }\{u^{\epsilon }(t,x+Fs)-u^{\epsilon }(t,x)\}\nu
(ds)y+y\mathbb{E}[R_{0}] \\
& =y\mathbb{E}[R_{0}],
\end{align*}%
where
\begin{equation*}
R_{0}=R_{1}(1+\theta c)+R_{2},
\end{equation*}%
\begin{equation*}
R_{2}=R_{21}+R_{22},
\end{equation*}%
and
\begin{align*}
R_{21}& =\theta ^{2}\Big(D_{1}u^{\epsilon }+D^{1}u^{\epsilon }[b-F\gamma
_{\epsilon }]+\frac{1}{2}(a+FB_{\epsilon }F^{T}):\nabla \nabla u^{\epsilon }%
\Big)c, \\
R_{22}& =\theta \mathbf{I}(\theta <h)[u^{\epsilon }(t,x+FJ)-u^{\epsilon
}(t,x)]c.
\end{align*}

It is clear that many of the terms in $R$ are only non--zero in the case $%
\theta <h$, i.e. when a jump occurs. We rearrange the terms in $R_{0}$
according to their degree in $\theta$:
\begin{align*}
R_{0}& =\underbrace{R_{17}+R_{18}+R_{19}+R_{22}}_{\text{$\mathbf{I}(\theta
<h)\theta $-terms}}\ +\underbrace{%
R_{11}+R_{12}+R_{13}+R_{14}+R_{15}+R_{16}+R_{21}}_{\text{$\theta ^{2}$ -
terms}} \\
& \quad +\underbrace{\theta c(R_{17}+R_{18}+R_{19})}_{\text{$(\mathbf{I}%
(\theta <h)\theta ^{2}$-terms}}\ +\underbrace{\theta
c(R_{11}+R_{12}+R_{13}+R_{14}+R_{15}+R_{16})}_{\text{$\theta ^{3}$ - terms}}
\end{align*}

Now to estimate the terms in the error $R_{0},$ we observe that (i) $%
\int_{|s|>\epsilon }s\nu (ds)=\gamma _{\epsilon }+\int_{|s|>1}s\nu (ds)$
with the latter integral bounded and, in particular, $|\mathbb{E}[J]|\leq
K(1+|\gamma _{\epsilon }|)/\lambda _{\epsilon };$ (ii) $\mathbb{E}[J]^{2p}$,
$p\geq 1,$ are bounded; (iii) the terms $R_{17}$, $R_{18}$, $R_{19},$ $%
R_{21} $ and $R_{22}$ contain derivatives of $u^{\epsilon }$ evaluated at or
between the points $Q_{3}$ and $Q_{4}$ and in their estimation Assumption~%
\ref{assum:differentiability} and \eqref{bQ2} from Lemma~\ref%
{lem:growth-intermediate-points} are used; (iv) the terms $R_{11}$, $R_{12}$%
, $R_{13}$, $R_{14}$, $R_{15}$ and $R_{16}$ contain derivatives of $%
u^{\epsilon }$ evaluated at or between the points $Q_{1}$ and $Q_{2}$ and in
their estimation Assumption~\ref{assum:differentiability}, \eqref{newbQ1}
from Lemma~\ref{lem:growth-intermediate-points}, and Lemma~\ref{lem:momth}
are used; (v) $\gamma _{\epsilon }^{2}/\lambda _{\epsilon }$ is bounded by a
constant independent of $\epsilon .$ As a result, we obtain%
\begin{align*}
\Big|& \mathbb{E}\big[R_{17}+R_{18}+R_{19}+R_{22}\big]\Big|\leq K_{1}\frac{%
(1+\gamma _{\epsilon }^{2})}{\lambda _{\epsilon }}\mathbb{E}\left[ \mathbf{I}%
(\theta <h)\theta \right] , \\
\Big|& \mathbb{E}\big[\theta (R_{17}+R_{18}+R_{19})\big]\Big|\leq K_{2}\frac{%
(1+\gamma _{\epsilon }^{2})}{\lambda _{\epsilon }}\mathbb{E}\left[ \mathbf{I}%
(\theta <h)\theta ^{2}\right] \leq K_{3}\frac{(1+\gamma _{\epsilon }^{2})}{%
\lambda _{\epsilon }}\mathbb{E}\left[ \mathbf{I}(\theta <h)\theta \right] ,
\\
\Big|& \mathbb{E}\big[(R_{11}+R_{12}+R_{13}+R_{14}+R_{15}+R_{16}+R_{21})\big]%
\Big|\leq K_{4}(1+\gamma _{\epsilon }^{2})(\mathbb{E}\left[ \theta ^{2}%
\right] +\gamma _{\epsilon }^{q}\mathbb{E}\theta ^{q+2})) \\
& \leq K_{5}(1+\gamma _{\epsilon }^{2})\frac{1-e^{-\lambda _{\epsilon
}h}(1+\lambda _{\epsilon }h)}{\lambda _{\epsilon }^{2}},
\end{align*}%
and%
\begin{align*}
\Big|& \mathbb{E}\big[\theta (R_{11}+R_{12}+R_{13}+R_{14}+R_{15}+R_{16})\big]%
\Big|\leq K_{6}(1+\gamma _{\epsilon }^{2})(\mathbb{E}\left[ \theta ^{3}%
\right] +\gamma _{\epsilon }^{q}\mathbb{E}\theta ^{q+3})) \\
& \leq K_{7}(1+\gamma _{\epsilon }^{2})\frac{1-e^{-\lambda _{\epsilon
}h}(1+\lambda _{\epsilon }h)}{\lambda _{\epsilon }^{3}}\leq K_{8}(1+\gamma
_{\epsilon }^{2})\frac{1-e^{-\lambda _{\epsilon }h}(1+\lambda _{\epsilon }h)%
}{\lambda _{\epsilon }^{2}},
\end{align*}%
where all constants $K_{i}>0$ are independent of $h$ and $\epsilon $ and $%
q\geq 1$.

Overall we obtain
\begin{align*}
\Big|\mathbb{E}[R]\Big|& \leq (K_{1}+K_{3})\frac{(1+\gamma _{\epsilon }^{2})%
}{\lambda _{\epsilon }}y\mathbb{E}\left[ \mathbf{I}(\theta <h)\theta \right]
+(K_{5}+K_{8})(1+\gamma _{\epsilon }^{2})y\frac{1-e^{-\lambda _{\epsilon
}h}(1+\lambda _{\epsilon }h)}{\lambda _{\epsilon }^{2}} \\
& \leq K\left\{ \frac{1}{\lambda _{\epsilon }}\mathbb{E}\left[ \mathbf{I}%
(\theta <h)\theta \right] +\frac{1-e^{-\lambda _{\epsilon }h}(1+\lambda
_{\epsilon }h)}{\lambda _{\epsilon }^{2}}\right\} (1+\gamma _{\epsilon
}^{2})y \\
& =2K(1+\gamma _{\epsilon }^{2})\frac{1-e^{-\lambda _{\epsilon }h}(1+\lambda
_{\epsilon }h)}{\lambda _{\epsilon }^{2}}y.\qedhere
\end{align*}%
$\qedhere$
\end{proof}

\begin{remark}
\label{Rem:onestep}We note the following two asymptotic regimes for the
one-step error (\ref{one-step}). For $\lambda _{\epsilon }h<1$ (in practice,
this occurs only when $\lambda _{\epsilon }$ is small or moderate like it is
in jump-diffusions), we can expand the exponent in (\ref{one-step}) and
obtain that the one-step error is of order $O(h^{2}):$
\begin{equation*}
\big|\mathbb{E}[R(t,x,y,z)]\big|\leq K(1+\gamma _{\epsilon }^{2})h^{2}y.
\end{equation*}

When $\lambda _{\epsilon }$ is very large (e.g., for small $\epsilon $ in
the infinite activity case) then the term with $e^{-\lambda _{\epsilon }h}$
can be neglected and we get
\begin{equation*}
\big|\mathbb{E}[R(t,x,y,z)]\big|\leq K\frac{1+\gamma _{\epsilon }^{2}}{%
\lambda _{\epsilon }^{2}}y.
\end{equation*}%
The usefulness of a more precise estimate (\ref{one-step}) is that it
includes situations in between these two asymptotic regimes and also allows
to consider an interplay between $h$ and $\epsilon $ (see Section~\ref%
{sec:iij}).
\end{remark}

\subsection{Global error\label{sec:glob}}

In this section we obtain an estimate for the global weak-sense error of
Algorithm~\ref{Hca01}. We first estimate average number of steps $\mathbb{E}%
\varkappa $ of Algorithm~\ref{Hca01}.

\begin{lemma}[\textit{Number of steps}]
\label{lem:nums} The average number of steps $\varkappa $ for the chain $%
X_{k}$ from Algorithm~\ref{Hca01} satisfies the following bound
\begin{equation*}
E\varkappa \leq \frac{(T-t_{0})\lambda _{\epsilon }}{1-e^{-\lambda
_{\epsilon }h}}+1.
\end{equation*}
\end{lemma}

\begin{proof}
It is obvious that if we replace the bounded domain $G$ in Algorithm~\ref%
{Hca01} with the whole space $\mathbb{R}^{d}$ (i.e., replace the Dirichlet
problem by the Cauchy one), then the corresponding number of steps $%
\varkappa ^{\prime }$ of Algorithm~\ref{Hca01} is not less than $\varkappa .$
Hence it is sufficient to get an estimate for $E\varkappa ^{\prime }.$ Let $%
\delta _{1},\delta _{2},\dots $ be the interarrival times of the jumps, $%
\theta _{i}=\delta _{i}\wedge h$ for $i\geq 0,$ and $S_{k}=\sum_{i=0}^{k-1}%
\theta _{i}$ for $k\geq 0$. Then
\begin{equation*}
\varkappa \leq \varkappa ^{\prime }:=\inf \{l:S_{l}\geq T-t_{0}\}.
\end{equation*}%
Introduce the martingale: $\tilde{S}_{0}=0$ and $\tilde{S}_{k}:=S_{k}-k%
\mathbb{E}\theta $ for $k\geq 1$. Since $\theta _{i}\leq h$ we have that $%
\tilde{S}_{\varkappa ^{\prime }-1}\leq S_{\varkappa ^{\prime }-1}<T-t_{0}$
almost surely and thus by the optional stopping theorem we obtain
\begin{equation*}
\mathbb{E}\tilde{S}_{\varkappa ^{\prime }-1}=\mathbb{E}\tilde{S}_{0}=0.
\end{equation*}%
Therefore
\begin{equation*}
\mathbb{E}\tilde{S}_{\varkappa ^{\prime }-1}=\mathbb{E}[\varkappa ^{\prime
}-1]\cdot \mathbb{E}[\theta ]
\end{equation*}%
and we conclude
\begin{eqnarray*}
\mathbb{E}\varkappa &\leq &\mathbb{E}\varkappa ^{\prime }=\mathbb{E}%
[\varkappa ^{\prime }-1]+1 \\
&=&\frac{\mathbb{E}S_{\varkappa ^{\prime }-1}}{\mathbb{E}\theta }+1\leq
\frac{(T-t_{0})\lambda _{\epsilon }}{1-\mathrm{e}^{-\lambda _{\epsilon }h}}%
+1.\qedhere
\end{eqnarray*}
\end{proof}

We also need the following auxiliary lemma.

\begin{lemma}[Boundedness of $Y_{k}$ in Algorithm~\protect\ref{Hca01}]
\label{lem:boundednessY} The chain $Y_{k}$ defined in (\ref{Hc02}) is
uniformly bounded by a deterministic constant:
\begin{equation*}
Y_{k}\leq e^{\bar{c}(T-t_{0}+h)},
\end{equation*}%
where $\bar{c}=\max_{(t,x)\in \bar{Q}}c(t,x)$.
\end{lemma}

\begin{proof}
From (\ref{Hc02}), we can express $Y_{k}$ via previous $Y_{k-1}$ and get the
required estimate as follows:
\begin{align*}
Y_{k}& =Y_{k-1}(1+\theta _{k}c(t_{k-1},x_{k-1})\leq Y_{k-1}(1+\theta _{k}%
\bar{c}) \\
& \leq Y_{k-1}e^{\bar{c}\theta _{k-1}}\leq Y_{k-2}e^{\bar{c}(\theta
_{k}+\theta _{k-1})}\leq Y_{0}e^{\bar{c}(\vartheta _{k}-t_{0})}\leq e^{\bar{c%
}(T-t_{0}+h)}.\qedhere
\end{align*}
\end{proof}

Now we prove the convergence theorem for Algorithm~\ref{Hca01}.

\begin{theorem}[Global error of Algorithm~\protect\ref{Hca01}]
\label{thm:total-error} Under Assumption~\ref{assumption2.1'} with $l=2,$ $%
m=4$ and Assumptions~\ref{assumption3.1}, \ref{assum:differentiability} and %
\ref{assum:moments-jumps}, the global error of Algorithm~\ref{Hca01}
satisfies the following bound
\begin{align}
&\big|\mathbb{E}[\varphi (\bar{\vartheta}_{\varkappa },X_{\varkappa
})Y_{\varkappa }+Z_{\varkappa }]-u^{\epsilon }(t_{0},x_{0})\big|  \label{gl1}
\\
&\leq K(1+\gamma _{\epsilon }^{2})\left( \frac{1}{\lambda _{\epsilon }}-h%
\frac{e^{-\lambda _{\epsilon }h}}{1-e^{-\lambda _{\epsilon }h}}\right) +K%
\frac{1-e^{-\lambda _{\epsilon }h}}{\lambda _{\epsilon }},  \notag
\end{align}%
where $K>0$ is a constant independent of $h$ and $\epsilon .$
\end{theorem}

\begin{proof}
Recall (see (\ref{eq:fkac2})):
\begin{equation*}
u^{\epsilon }(t,x)=\mathbb{E}\left[ \varphi \left( \tilde{\tau}_{t,x},\tilde{%
X}_{t,x}(\tilde{\tau}_{t,x})\right) \tilde{Y}_{t,x,1}(\tilde{\tau}_{t,x})+%
\tilde{Z}_{t,x,1,0}(\tilde{\tau}_{t,x})\right] .
\end{equation*}%
The global error
\begin{equation*}
\mathbf{R}:=\big|\mathbb{E}[\varphi (\bar{\vartheta}_{\varkappa
},X_{\varkappa })Y_{\varkappa }+Z_{\varkappa }]-u^{\epsilon }(t_{0},x_{0})%
\big|
\end{equation*}%
can be written as
\begin{align}
\mathbf{R}& =\big|\mathbb{E}[\mathbf{I}(\vartheta _{\varkappa }\geq T)\left(
\varphi (\bar{\vartheta}_{\varkappa },X_{\varkappa })Y_{\varkappa
}-u^{\epsilon }(\vartheta _{\varkappa },X_{\varkappa })Y_{\varkappa }\right)
+u^{\epsilon }(\vartheta _{\varkappa },X_{\varkappa })Y_{\varkappa
}+Z_{\varkappa }-v(t_{0},x_{0})]\big|  \label{gl2} \\
& \leq \big|\mathbb{E}[\mathbf{I}(\vartheta _{\varkappa }\geq T)\left(
\varphi (\bar{\vartheta}_{\varkappa },X_{\varkappa })Y_{\varkappa
}-u^{\epsilon }(\vartheta _{\varkappa },X_{\varkappa })Y_{\varkappa }\right)
]\big|+\big|\mathbb{E}[u^{\epsilon }(\vartheta _{\varkappa },X_{\varkappa
})Y_{\varkappa }+Z_{\varkappa }-u^{\epsilon }(t_{0},x_{0})]\big|.  \notag
\end{align}%
Using Lemma~\ref{lem:boundednessY}, Assumption~\ref{assum:differentiability}
and Lemmas~\ref{lem:growth-intermediate-points} and~\ref{lem:momth} as well
as that $\bar{\vartheta}_{\varkappa }-\vartheta _{\varkappa }\leq \theta
_{\varkappa }$, we have for the first term in (\ref{gl2}):
\begin{equation}
\mathbb{E}[\mathbf{I}(\vartheta _{\varkappa }\geq T)\left( \varphi (\bar{%
\vartheta}_{\varkappa },X_{\varkappa })Y_{\varkappa }-u^{\epsilon
}(\vartheta _{\varkappa },X_{\varkappa })Y_{\varkappa }\right) ]\leq KE\left[
\theta _{\varkappa }(1+\gamma _{\epsilon }^{q}\theta _{\varkappa }^{q})%
\right] \leq K\frac{1-e^{-\lambda _{\epsilon }h}}{\lambda _{\epsilon }},
\label{gl3}
\end{equation}%
where $K>0$ does not depend on $h$ or $\varepsilon .$

For the second term in (\ref{gl2}), we exploit ideas from \cite{MT04} to
re-express the global error. We get using Theorem~\ref{Thm:Alg1onestep} and
Lemmas \ref{lem:boundednessY} and \ref{lem:nums}:%
\begin{align}
& \left\vert \mathbb{E}[u^{\epsilon }(\vartheta _{\varkappa },X_{\varkappa
})Y_{\varkappa }+Z_{\varkappa }-u^{\epsilon }(t_{0},x_{0})]\right\vert
\label{gl4} \\
& =\left\vert \mathbb{E}\left[ \sum\limits_{k=0}^{\varkappa -1}\mathbb{E}%
\left[ u^{\epsilon }(\vartheta _{k+1},X_{k+1})Y_{k+1}+Z_{k+1}-u^{\epsilon
}(\vartheta _{k},X_{k})Y_{k}-Z_{k}\Big|\vartheta _{k},X_{k},Y_{k},Z_{k}%
\right] \right] \right\vert  \notag \\
& =\left\vert \mathbb{E}\left[ \sum\limits_{k=0}^{\varkappa -1}\mathbb{E}%
\left[ R(\vartheta _{k},X_{k},Y_{k},Z_{k})\Big|\vartheta
_{k},X_{k},Y_{k},Z_{k}\right] \right] \right\vert  \notag \\
& \leq \mathbb{E}\left[ \sum\limits_{k=0}^{\varkappa -1}\frac{1-e^{-\lambda
_{\epsilon }h}(1+\lambda _{\epsilon }h)}{\lambda _{\epsilon }^{2}}K(1+\gamma
_{\epsilon }^{2})Y_{k}\right]  \notag \\
& \leq K\frac{1+\gamma _{\epsilon }^{2}}{\lambda _{\epsilon }^{2}}\left(
1-e^{-\lambda _{\epsilon }h}(1+\lambda _{\epsilon }h)\right) \mathbb{E}%
\varkappa  \notag \\
& \leq K(1+\gamma _{\epsilon }^{2})\left( \frac{1}{\lambda _{\epsilon
}(1-e^{-\lambda _{\epsilon }h)}}-h\frac{e^{-\lambda _{\epsilon }h}}{%
1-e^{-\lambda _{\epsilon }h}}\right) (T-t_{0})  \notag \\
& \leq K(1+\gamma _{\epsilon }^{2})\left( \frac{1}{\lambda _{\epsilon }}-h%
\frac{e^{-\lambda _{\epsilon }h}}{1-e^{-\lambda _{\epsilon }h}}\right) ,
\notag
\end{align}%
where, as usual constants $K>0$ are changing from line to line. Combining (%
\ref{gl2})-(\ref{gl4}), we arrive at (\ref{gl1}).
\end{proof}

\begin{remark}[Error estimate and convergence]
Note that the error estimate in Theorem~\ref{thm:total-error} gives us the
expected results in the limiting cases (see also Remark~\ref{Rem:onestep}).
If $\lambda _{\epsilon }h<1$, we obtain:
\begin{equation*}
\mathbf{R}\leq K(1+\gamma _{\epsilon }^{2})h,
\end{equation*}%
which is expected for weak convergence in the jump-diffusion case.

If $\lambda _{\epsilon }$ is large (meaning that almost always $\theta <h$),
the error is tending to
\begin{equation*}
\mathbf{R}\leq K(1+\gamma _{\epsilon }^{2})\frac{1}{\lambda _{\epsilon }},
\end{equation*}%
as expected (cf. \cite{Higa10}).

We also remark that for any fixed $\lambda _{\epsilon }$, we have first
order convergence when $h\rightarrow 0.$
\end{remark}

\begin{remark}
\label{rem:sym}In the case of symmetric measure $\nu (z)$ we have $\gamma
_{\epsilon }=0$ and hence the global error (\ref{gl1}) becomes
\begin{align}
&\big|\mathbb{E}[\varphi (\bar{\vartheta}_{\varkappa },X_{\varkappa
})Y_{\varkappa }+Z_{\varkappa }]-u^{\epsilon }(t_{0},x_{0})\big|
\label{sym1} \\
&\leq K\left( \frac{1}{\lambda _{\epsilon }}-h\frac{e^{-\lambda _{\epsilon
}h}}{1-e^{-\lambda _{\epsilon }h}}\right) +K\frac{1-e^{-\lambda _{\epsilon
}h}}{\lambda _{\epsilon }}.  \notag
\end{align}
\end{remark}

\subsection{Remark on the Cauchy problem\label{sec:Cau}}

Let us set $G=\mathbb{R}^{d}$ in (\ref{v1}) and hence consider the Cauchy
problem for the PIDE:%
\begin{align}
\frac{\partial u^{\epsilon }}{\partial t}+L_{\epsilon }u^{\epsilon
}+c(t,x)u^{\epsilon }+g(t,x)& =0, & & (t,x)\in Q,  \label{c1} \\
u^{\epsilon }(T,x)& =\varphi (t,x), & & x\in \mathbb{R}^{d}.  \notag
\end{align}%
In this case Algorithm~\ref{Hca01} stops only when $\vartheta _{\varkappa
}\geq T$ as there is no spatial boundary. Theorem~\ref{Thm:Alg1onestep}
remains valid for the Cauchy problem, although in this case one should
replace the constant $K$ in the right-hand side of the bound \eqref{one-step}
with a function $K(x)>0$ satisfying
\begin{equation*}
K(x)\leq \tilde{K}(1+|x|^{2q})
\end{equation*}%
with some constants $\tilde{K}>0$ and $q\geq 1.$ Consequently, to prove an
analogue of the global convergence Theorem~\ref{thm:total-error}, we need to
prove boundedness of moments $\mathbb{E}X_{k}^{2p}.$ Let
\begin{equation*}
X_{k}\equiv X_{\varkappa }\ \ \text{for all }k\geq \varkappa .
\end{equation*}

\begin{lemma}
\label{lem:Cmom}Under Assumptions~\ref{assumption3.1}, \ref{assumption3.2},
and~\ref{assum:moments-jumps}, we have for $X_{k}$ from Algorithm~\ref{Hca01}%
:
\begin{equation}
\mathbb{E}|X_{k}|^{2p}\leq K(1+|x|^{2p})  \label{c2}
\end{equation}%
with some constants $K>0$ and $p\geq 1.$
\end{lemma}

\begin{proof}
As usual, in this proof $K>0$ is a constant independent of $\epsilon $ and $%
h $ which can change from line to line in derivations. We first prove the
lemma for an integer $p\geq 1.$

We have
\begin{align}
& |X_{k+1}|^{2p}=|(X_{k+1}-X_{k})+X_{k}|^{2p}\leq |X_{k}|^{2p}+\mathbb{I}%
_{\varkappa \geq k+1}\left\vert X_{k}\right\vert ^{2p-2}  \label{c3} \\
& \times \left[ 2p(X_{k},X_{k+1}-X_{k})+p(2p-1)|X_{k+1}-X_{k}|^{2}\right]
+K\sum_{l=3}^{2p}\left\vert X_{k}\right\vert ^{2p-l}|X_{k+1}-X_{k}|^{l}.
\notag
\end{align}%
For $\varkappa >k:$%
\begin{align*}
X_{k+1}-X_{k}& =\theta _{k+1}\left( b(\vartheta _{k},X_{k})-F(\vartheta
_{k},X_{k})\gamma _{\epsilon }\right) +\sqrt{\theta _{k+1}}\left( \sigma
(\vartheta _{k},X_{k})\xi _{k}+F(\vartheta _{k},X_{k})\beta _{\epsilon }\eta
\right) \\
& \qquad +\mathbf{I}(\delta _{k+1}<h)F(\vartheta _{k},X_{k})J_{\epsilon
,k+1}.
\end{align*}%
Then
\begin{align*}
\mathbb{E}\left( X_{k+1}-X_{k}|X_{k}\right) & =\mathbb{I}_{\varkappa
>k}\left( b(\vartheta _{k},X_{k})-F(\vartheta _{k},X_{k})\gamma _{\epsilon
}\right) \mathbb{E}\theta _{k+1} \\
& +\mathbb{I}_{\varkappa >k}F(\vartheta _{k},X_{k})\mathbb{E}\left( \mathbf{I%
}(\delta _{k+1}<h)J_{\epsilon ,k+1}\right) \\
& =\mathbb{I}_{\varkappa >k}\frac{1-e^{-\lambda _{\epsilon }h}}{\lambda
_{\epsilon }}\left[ b(\vartheta _{k},X_{k})+F(\vartheta
_{k},X_{k})\int_{|s|>1}s\nu (ds)\right] ,
\end{align*}%
where we used%
\begin{align*}
& -\gamma _{\epsilon }\mathbb{E}\theta _{k+1}+\mathbb{E}\left( \mathbf{I}%
(\delta _{k+1}<h)J_{\epsilon ,k+1}\right) \\
& =-\gamma _{\epsilon }\frac{1-e^{-\lambda _{\epsilon }h}}{\lambda
_{\epsilon }}+\left( 1-e^{-\lambda _{\epsilon }h}\right) \left[ \frac{\gamma
_{\epsilon }}{\lambda _{\epsilon }}+\frac{1}{\lambda _{\epsilon }}%
\int_{|s|>1}s\nu (ds)\right] \\
& =\frac{1-e^{-\lambda _{\epsilon }h}}{\lambda _{\epsilon }}\int_{|s|>1}s\nu
(\mathrm{d}s).
\end{align*}%
By the linear growth Assumption~\ref{assumption3.2} and Assumption~\ref%
{assum:moments-jumps}, we get
\begin{align}
& \left\vert \mathbb{E}\left[ \left\vert X_{k}\right\vert
^{2p-2}(X_{k},X_{k+1}-X_{k})\right] \right\vert  \label{c4} \\
& \leq K\frac{1-e^{-\lambda _{\epsilon }h}}{\lambda _{\epsilon }}\left(
\mathbb{E}\mathbf{I}_{\varkappa >k}(\omega )|X_{k}|^{2p-2}+\mathbb{E}\mathbf{%
I}_{\varkappa >k}(\omega )|X_{k}|^{2p}\right)  \notag \\
& \leq K\frac{1-e^{-\lambda _{\epsilon }h}}{\lambda _{\epsilon }}\mathbb{E}%
\mathbf{I}_{\varkappa >k}(\omega )|X_{k}|^{2p}  \notag
\end{align}%
using that $\mathbb{E}\mathbf{I}_{\varkappa >k}(\omega )|X_{k}|^{2p-2}\leq
\mathbb{E}\mathbf{I}_{\varkappa >k}(\omega )|X_{k}|^{2p}$. Further,
\begin{eqnarray*}
\mathbb{E}\left( |X_{k+1}-X_{k}|^{2}|X_{k}\right) &\leq &\mathbf{I}%
_{\varkappa >k}\left( a(\vartheta _{k},X_{k})+F(\vartheta
_{k},X_{k})B_{\epsilon }F^{\top }(\vartheta _{k},X_{k})\right) \mathbb{E}%
\theta _{k+1} \\
&&+2\mathbf{I}_{\varkappa >k}\left( b(\vartheta _{k},X_{k})-F(\vartheta
_{k},X_{k})\gamma _{\epsilon }\right) ^{2}\mathbb{E}\theta _{k+1}^{2} \\
&&+2\mathbf{I}_{\varkappa >k}F(\vartheta _{k},X_{k})F^{\top }(\vartheta
_{k},X_{k})\mathbb{E}\left( \mathbf{I}(\delta _{k+1}<h)J_{\epsilon
,k+1}^{2}\right) ,
\end{eqnarray*}%
and thus
\begin{eqnarray}
\left\vert \mathbb{E}\left[ \left\vert X_{k}\right\vert
^{2p-2}|X_{k+1}-X_{k}|^{2}\right] \right\vert &\leq &K\frac{1-e^{-\lambda
_{\epsilon }h}}{\lambda _{\epsilon }}\left( \mathbb{E}\mathbf{I}_{\varkappa
>k}|X_{k}|^{2p-2}+\mathbb{E}\mathbf{I}_{\varkappa >k}|X_{k}|^{2p}\right)
\label{c5} \\
&&+K\frac{1-e^{-\lambda _{\epsilon }h}(1+\lambda _{\epsilon }h)}{\lambda
_{\epsilon }^{2}}(1+\gamma _{\epsilon }^{2})\left( \mathbb{E}\mathbf{I}%
_{\varkappa >k}|X_{k}|^{2p-2}+\mathbb{E}\mathbf{I}_{\varkappa
>k}|X_{k}|^{2p}\right)  \notag \\
&&+K\frac{1-e^{-\lambda _{\epsilon }h}}{\lambda _{\epsilon }}\left( \mathbb{%
\ E}\mathbf{I}_{\varkappa >k}|X_{k}|^{2p-2}+\mathbb{E}\mathbf{I}_{\varkappa
>k}|X_{k}|^{2p}\right)  \notag \\
&\leq &K\frac{1-e^{-\lambda _{\epsilon }h}}{\lambda _{\epsilon }}\mathbb{E}%
\mathbf{I}_{\varkappa >k}|X_{k}|^{2p},  \notag
\end{eqnarray}%
using that
\begin{equation*}
\mathbb{E}J_{\epsilon ,k+1}^{2}=\frac{1}{\lambda _{\epsilon }}%
\int_{|s|>\epsilon }s^{2}\nu (ds)=\frac{B_{\epsilon }}{\lambda _{\epsilon }}+%
\frac{1}{\lambda _{\epsilon }}\int_{|s|>1}s^{2}\nu (ds)
\end{equation*}%
and that $\dfrac{\gamma _{\epsilon }^{2}}{\lambda _{\epsilon }}$ is bounded.

For the last term in (\ref{c3}), observe that
\begin{equation*}
\left\vert X_{k}\right\vert ^{2p-l}|X_{k+1}-X_{k}|^{l}\leq \theta _{k+1}
\left[ \frac{2p-l}{2p}|X_{k}|^{2p}+\frac{l}{2p}|X_{k+1}-X_{k}|^{2p}\frac{1}{%
\theta _{k+1}^{2p/l}}\right] ,\ l=3,\ldots ,2p.
\end{equation*}%
Then one can show that%
\begin{equation}
\sum_{l=3}^{2p}\mathbb{E}\left\vert X_{k}\right\vert
^{2p-l}|X_{k+1}-X_{k}|^{l}\leq K\frac{1-e^{-\lambda _{\epsilon }h}}{\lambda
_{\epsilon }}\mathbb{E}\mathbf{I}_{\varkappa >k}|X_{k}|^{2p}.  \label{c6}
\end{equation}

Combining (\ref{c3})-(\ref{c6}), we get%
\begin{equation*}
\mathbb{E}|X_{k+1}|^{2p}\leq \mathbb{E}|X_{k}|^{2p}+K\frac{1-e^{-\lambda
_{\epsilon }h}}{\lambda _{\epsilon }}\mathbb{E}\mathbf{I}_{\varkappa
>k}|X_{k}|^{2p}=\mathbb{E}|X_{k}|^{2p}+K\mathbb{E}\left[ \mathbb{\theta }%
_{k+1}\mathbb{I}_{\varkappa >k}|X_{k}|^{2p}\right] ,
\end{equation*}%
whence
\begin{equation}
\mathbb{E}|X_{\varkappa }|^{2p}\leq |x_{0}|^{2p}+K\mathbb{E}%
\dsum\limits_{k=0}^{\varkappa -1}\mathbb{\theta }_{k+1}|X_{k}|^{2p}.
\label{c7}
\end{equation}%
Introduce a continuous time piece-wise constant process
\begin{equation*}
\tilde{U}(t)=|X_{k}|^{2p}\ \ \text{for }t\in \lbrack \vartheta
_{k},\vartheta _{k+1}),\ \ k=0,\ldots ,\varkappa -1,
\end{equation*}%
and
\begin{equation*}
\tilde{U}(t)=|X_{\varkappa }|^{2p}\text{ for }t\geq \vartheta _{\varkappa }.
\end{equation*}%
Then we can write (\ref{c7}) as%
\begin{eqnarray*}
E\tilde{U}(\vartheta _{\varkappa }) &=&E\tilde{U}(T+h)\leq
|x_{0}|^{2p}+KE\int_{t_{0}}^{\vartheta _{\varkappa }}\tilde{U}(t)ds \\
&\leq &|x_{0}|^{2p}+K\int_{t_{0}}^{T+h}E\tilde{U}(t)ds.
\end{eqnarray*}%
By Gronwall's inequality, we get
\begin{equation*}
E\tilde{U}(\vartheta _{\varkappa })\leq e^{K(T+h-t_{0})}|x_{0}|^{2p},
\end{equation*}%
implies (\ref{c2}) for integer $p\geq 1$. Then, by Jensen's inequality, (\ref%
{c2}) holds for non-integer $p\geq 1$ as well.
\end{proof}

Based on the discussion before Lemma~\ref{lem:Cmom} and on the moments
estimate (\ref{c2}) of Lemma~\ref{lem:Cmom}, it is not difficult to show
that the global error estimate (\ref{gl1}) for Algorithm~\ref{Hca01} also
holds in the Cauchy problem case.

\subsection{The case of infinite intensity of jumps\label{sec:iij}}

In this section we combine the previous results, Theorem~\ref{Thm_eps}\ and~%
\ref{thm:total-error}, to obtain an overall error estimate for solving the
problem (\ref{eq:problem}) in the case of infinite intensity of jumps by
Algorithm~\ref{Hca01}. We obtain%
\begin{align}
&\big|\mathbb{E}[\varphi (\bar{\vartheta}_{\varkappa },X_{\varkappa
})Y_{\varkappa }+Z_{\varkappa }]-u(t_{0},x_{0})\big|  \label{inf1} \\
&\leq K(1+\gamma _{\epsilon }^{2})\left( \frac{1}{\lambda _{\epsilon }}-h%
\frac{e^{-\lambda _{\epsilon }h}}{1-e^{-\lambda _{\epsilon }h}}\right) +K%
\frac{1-e^{-\lambda _{\epsilon }h}}{\lambda _{\epsilon }}+K\int_{|z|\leq
\epsilon }|z|^{3}\nu (\mathrm{d}z),  \notag
\end{align}%
where $K>0$ is independent of $h$ and $\epsilon .$

Let us consider an $\alpha $-stable process as in Example~\ref%
{ex:alpha-stable}, i.e., for $\alpha \in (0,2)$ the L\'evy measure
\begin{equation*}
\nu (\mathrm{d}z)\sim |z|^{-1-\alpha }\mathrm{d}z,
\end{equation*}%
where we are focusing our attention on the singularity near zero. Then
\begin{gather*}
\lambda _{\epsilon }=\int_{|z|\geq \epsilon }\nu (\mathrm{d}z)\sim \epsilon
^{-\alpha }, \\
\gamma _{\epsilon }^{2}=\left[ \int_{\epsilon \leq |z|\leq 1}z^{i}\nu (%
\mathrm{d}z)\right] ^{2}\sim \epsilon ^{2-2\alpha }, \\
\int_{|z|\leq \epsilon }|z|^{3}\nu (\mathrm{d}y)\sim \epsilon ^{3-\alpha }.
\end{gather*}%
Hence
\begin{align}
&\big|\mathbb{E}[\varphi (\bar{\vartheta}_{\varkappa },X_{\varkappa
})Y_{\varkappa }+Z_{\varkappa }]-u(t_{0},x_{0})\big| \\
&\leq K\left[ (1+\epsilon ^{2-2\alpha })\left( \epsilon ^{\alpha }-h\frac{%
e^{-\epsilon ^{-\alpha }h}}{1-e^{-\epsilon ^{-\alpha }h}}\right) +\epsilon
^{\alpha }\left( 1-e^{-\epsilon ^{-\alpha }h}\right) +\epsilon ^{3-\alpha }%
\right] .  \notag
\end{align}
Let us measure the computational cost of Algorithm~\ref{Hca01} in terms of
the average number of steps (see Lemma~\ref{lem:nums}). Since
\begin{equation*}
E\varkappa \leq \frac{(T-t_{0})\lambda _{\epsilon }}{1-e^{-\lambda
_{\epsilon }h}}\leq K\frac{\epsilon ^{-\alpha }}{1-e^{-\epsilon ^{-\alpha }h}%
},
\end{equation*}%
we choose to use the cost associated with the average number of steps as
\begin{equation*}
C:=\frac{\epsilon ^{-\alpha }}{1-e^{-\epsilon ^{-\alpha }h}}.
\end{equation*}%
We fix a tolerance level $\rho _{tol}$ and require $\epsilon $ and $h$ to be
so that
\begin{equation*}
\rho _{tol}=\rho (\epsilon ,h):=(1+\epsilon ^{2-2\alpha })\left( \epsilon
^{\alpha }-\frac{he^{-\epsilon ^{-\alpha }h}}{1-e^{-\epsilon ^{-\alpha }h}}%
\right) +\epsilon ^{\alpha }\left( 1-e^{-\epsilon ^{-\alpha }h}\right)
+\epsilon ^{3-\alpha }.
\end{equation*}%
Note that since we are using the Euler scheme for SDE approximation, the
decrease of $\rho _{tol}$ in terms of cost cannot be faster than linear. We
now consider three cases of $\alpha .$

\textbf{The case }$\alpha =1.$ We have
\begin{equation*}
\rho (\epsilon ,h)=2\left( \epsilon -\frac{he^{-\epsilon ^{-1}h}}{%
1-e^{-\epsilon ^{-1}h}}\right) +\epsilon \left( 1-e^{-\epsilon
^{-1}h}\right) +\epsilon ^{2}=O(\epsilon )
\end{equation*}%
and by choosing sufficiently small $\epsilon $ we can reach the required $%
\rho _{tol}.$ It is optimal to take $h=\infty $ (in practice, taking $%
h=T-t_{0})$ and the cost is then $C=1/\epsilon .$ Hence $\rho _{tol}$ is
inversely proportional to $C,$ and convergence is linear in cost (to reduce $%
\rho _{tol}$ twice, we need to double $C).$

\textbf{The case }$\alpha \in (0,1).$ We have
\begin{equation*}
\rho (\epsilon ,h)\leq \epsilon ^{2-\alpha }+2\epsilon ^{\alpha }+\epsilon
^{3-\alpha }=O(\epsilon ^{\alpha }).
\end{equation*}%
Again, it is optimal to take $h=\infty $ and we have linear convergence in
cost.

\textbf{The case }$\alpha \in (1,2).$ If we take $h=\infty ,$ then $\rho
(\epsilon ,h)=O(\epsilon ^{2-\alpha })$ and the convergence order in terms
of cost is $2/\alpha -1,$ which is very slow (e.g., for $\alpha =3/2,$ the
order is $1/3$ and for $\alpha =1.9,$ the order is $\approx 0.05$). Let us
now take $h=\epsilon ^{\ell }$ with $\ell \geq \alpha .$ Then
\begin{equation*}
\rho (\epsilon ,h)\leq \epsilon ^{2-2\alpha }h+2h+\epsilon ^{3-\alpha
}=\epsilon ^{2-2\alpha +\ell }+\epsilon ^{\ell }+\epsilon ^{3-\alpha }
\end{equation*}%
and $C\approx 1/h=\epsilon ^{-\ell }.$ The optimal $\ell =1+\alpha ,$ for
which $\rho (\epsilon ,h)=O(\epsilon ^{3-\alpha })$ and the convergence
order in terms of cost is $(3-\alpha )/(1+\alpha ),$ which is much better
(e.g., for $\alpha =3/2,$ the order is $3/5$ and it cannot be smaller than $%
1/3$ for any $\alpha \in (1,2)$). Note that in the case of symmetric measure
$\nu (z)$ (see Remark~\ref{rem:sym}), convergence is linear in cost for $%
\alpha \in (1,2).$\medskip\

To conclude, for $\alpha \in (0,1]$ we have first order convergence and
there is no benefit of restricting jump adapted steps by $h$ (see a similar
result in the case of the Cauchy problem and not restricted jump-adapted
steps in \cite{KOT13}). However, in the case of $\alpha \in (1,2),$ it is
beneficial to use restricted jump-adapted steps to get the order of $%
(3-\alpha )/(1+\alpha ).$ We also recall that restricted jump-adapted steps
should typically be used for jump-diffusions (the finite activity case when
there is no singularity of $\lambda _{\epsilon }$ and $\gamma _{\epsilon }$)
because jump time increments $\delta $ typically take too large values and
to control the error at every step we should truncate those times at a
sufficiently small $h>0$ for a satisfactory accuracy.

\section{Numerical experiments\label{sec:num}}

In this section we illustrate the theoretical results of Section~\ref%
{sec:weak}. In particular, we display the behaviour in the case of infinite
intensity of jumps for different regimes of $\alpha $. We showcase numerical
tests of Algorithm~\ref{Hca01} in three different examples: (i) a
non-singular L\'evy measure (Example~\ref{ex:non-singular}), (ii) a singular
L\'evy measure which is similar to that of Example~\ref{ex:alpha-stable}
(see Example~\ref{ex:singular-measure}), and (iii) pricing a
foreign-exchange (FX) barrier basket option where the underlying model is of
exponential L\'evy-type (Example~\ref{ex:finance}).

As it is usual for weak approximation (see e.g. \cite{MT04}), in simulations
we compliment Algorithm~\ref{Hca01} by the Monte Carlo techniques and
evaluate $u(t_{0},x)$ or $u^{\epsilon }(t_{0},x)$ as
\begin{equation}
\bar{u}(t_{0},x):=\mathbb{E}\left[ \varphi (\bar{\vartheta}_{\varkappa
},X_{\varkappa })Y_{\varkappa }+Z_{\varkappa }\right] \simeq \hat{u}=\frac{1%
}{M}\sum\limits_{m=1}^{M}\left[ \varphi (\bar{\vartheta}_{\varkappa
}^{(m)},X_{\varkappa }^{(m)})Y_{\varkappa }^{(m)}+Z_{\varkappa }^{(m)}\right]
,  \label{eq:MC}
\end{equation}%
where $(\bar{\vartheta}_{\varkappa }^{(m)},X_{\varkappa }^{(m)},Y_{\varkappa
}^{(m)},Z_{\varkappa }^{(m)})$ are independent realisations of $(\bar{%
\vartheta}_{\varkappa },X_{\varkappa },Y_{\varkappa },Z_{\varkappa })$. The
Monte Carlo error of (\ref{eq:MC}) is
\begin{equation*}
\sqrt{D_{M}}:=\frac{(\text{Var}f(\bar{X}(T)))^{1/2}}{M^{1/2}}\simeq \sqrt{%
\bar{D}_{M}}\,,
\end{equation*}%
where
\begin{equation*}
\bar{D}_{M}=\frac{1}{M}\left[ \frac{1}{M}\sum_{m=1}^{M}\left[ \Xi ^{(m)}%
\right] ^{2}-\left( \frac{1}{M}\sum_{m=1}^{M}\Xi ^{(m)}\right) ^{2}\right] ,
\end{equation*}%
and $\Xi ^{(m)}=\varphi \left( \bar{\vartheta}_{\varkappa
}^{(m)},X_{\varkappa }^{(m)}\right) Y_{\varkappa }^{(m)}+Z_{\varkappa
}^{(m)}.$ Then $\bar{u}(t_{0},x)$ falls in the corresponding confidence
interval $\hat{u}\pm 2\sqrt{\bar{D}_{M}}$ with probability $0.95$.

\subsection{Example with a non-singular L\'evy measure}

\label{ssec:exa-non-singular} In this subsection, we illustrate Algorithm~%
\ref{Hca01} in the case of a simple non-singular L\'{e}vy measure (i.e., the
jump-diffusion case), where there is no need to replace small jumps and
hence we directly approximate $u(t_{0},x)$ rather than $u^{\epsilon
}(t_{0},x).$ Consequently, the numerical integration error does not depend
on $\epsilon $. We recall (see Theorem~\ref{thm:total-error}) that Algorithm~%
\ref{Hca01} has first order of convergence in $h$.

\begin{example}[Non-singular L\'{e}vy measure]
\label{ex:non-singular} To construct this and the next example, we use the
same recipe as in \cite{MT03,MT04}: we choose the coefficients of the
problem \eqref{eq:problem} so that we can write down its solution
explicitly. Having the exact solution is very useful for numerical tests.

Consider the problem \eqref{eq:problem} with $d=3,$ $G=U_{1}$ which is the
open unit ball centred at the origin in $\mathbb{R}^{3},$ and with the
coefficients
\begin{align}
a^{11}(t,x)& =1.21-x_{2}^{2}-x_{3}^{2},\quad a^{22}=1,\quad a^{33}=1,
\label{ex_ab} \\
a^{ij}& =0,\ i\neq j,  \notag \\
b& =0,  \notag
\end{align}%
\begin{equation}
F(t,x)=(f,f,f)^{T},\ f\in \mathbb{R},  \label{ex_F}
\end{equation}%
\begin{align}
g(t,x)& :=\frac{1}{2}e^{T-t}(1.21-x_{1}^{4}-x_{2}^{4})+6(1-\frac{1}{2}%
e^{T-t})\left[ x_{1}^{2}(1.21-x_{2}^{2}-x_{3}^{2})+x_{2}^{2}\right]
\label{ex_g} \\
& \quad +(1-\frac{1}{2}e^{T-t})\Big[(C_{+}-C_{-})\frac{4f}{\mu ^{2}}%
(x_{1}^{3}+x_{2}^{3})+(C_{+}+C_{-})\frac{12f^{2}}{\mu ^{3}}%
(x_{1}^{2}+x_{2}^{2})  \notag \\
& \quad +(C_{+}-C_{-})\frac{24f^{3}}{\mu ^{4}}(x_{1}+x_{2})+(C_{+}+C_{-})%
\frac{48f^{4}}{\mu ^{5}}\Big],  \notag
\end{align}%
with the boundary condition
\begin{equation}
\varphi (t,x)=(1-\tfrac{1}{2}\mathrm{e}^{T-t})(1.21-x_{1}^{4}-x_{2}^{4})
\label{ex_fi}
\end{equation}%
and with the L\'{e}vy measure density
\begin{equation*}
\nu (\mathrm{d}z)=%
\begin{cases}
C_{-}\mathrm{e}^{-\mu |z|}\mathrm{d}z, & \text{if $z<0$}, \\
C_{+}\mathrm{e}^{-\mu |z|}\mathrm{d}z, & \text{if $z>0$},%
\end{cases}%
\end{equation*}%
where $C_{-}$ and $C_{+}$ are some positive constants.

It is not difficult to verify that this problem has the solution
\begin{equation*}
u(t,x)=(1-\tfrac{1}{2}\mathrm{e}^{T-t})(1.21-x_{1}^{4}-x_{2}^{4}).
\end{equation*}%
and we also find
\begin{align*}
\lambda & =\int_{|z|>0}\nu (\mathrm{d}z)=\int_{\mathbb{R}}\nu (\mathrm{d}z)=%
\frac{C_{+}+C_{-}}{\mu }, \\
\rho (z)& =\frac{C_{-}\mathrm{e}^{-\mu |z|}\mathbf{I}_{\{z<0\}}+C_{+}\mathrm{%
e}^{-\mu |z|}\mathbf{I}_{\{z>0\}}}{\lambda }.
\end{align*}%
We simulated jump sizes by analytically inverting the cumulative
distribution function corresponding to the density $\rho (z)$ and making use
of uniform random numbers in the standard manner.

\begin{table}[h]
\centering
\begin{minipage}{.45\textwidth}
	  \centering
	  \includegraphics[width=1.1\linewidth]{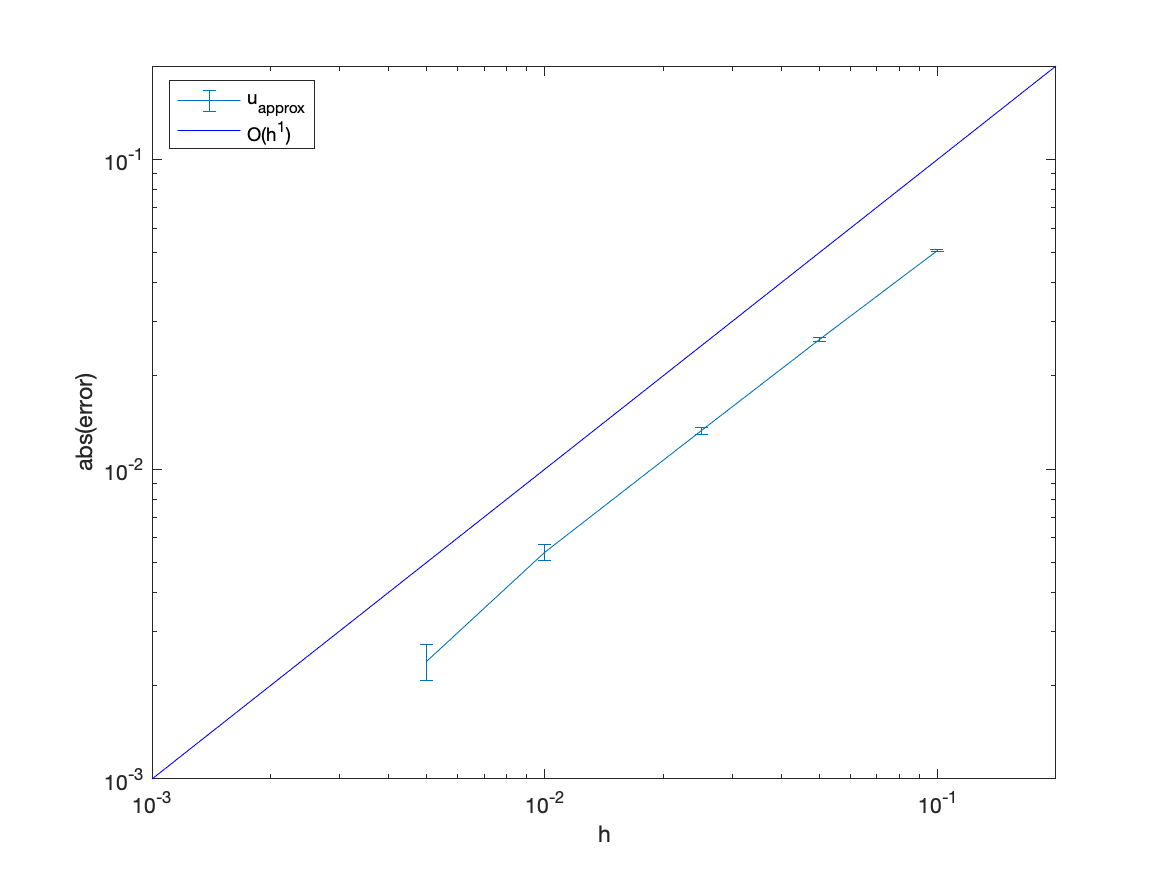}
	  \captionof{figure}{Non-singular L\'evy measure example: dependence of the error $e$ on $h$,
the error bars show the Monte Carlo error. The parameters used are
$T=1, C_{+}=30, C_{-}=1.0, \mu = 3.0, f=0.1, M = 40 000 000$ and $\hat u$ is evaluated at the point $(0,0)$.}
	  \label{fig:Ex1}
	\end{minipage}
\begin{minipage}{0.5\linewidth}
\caption{Non-singular L\'evy measure example. The parameters are the same as in Figure~\ref{fig:Ex1}.
The column $\hat{\varkappa}$ gives the sample average of the number of steps together with its Monte Carlo error.}
		\centering
		\vfill
		\begin{tabular}{ccccc}
			\hline
			$h$ & $\hat u$ & $2\sqrt{\hat D_M}$  & $e$  & $\hat{\varkappa}$\\ \hline
			0.1                   & 0.9367       & 0.0004 & 0.0507 & $7.72\pm 0.0037$   \\ \hline
			0.05                 & 0.9612       & 0.0004 & 0.0262 & $11.04\pm 0.0056$   \\ \hline
			0.025               & 0.9742       & 0.0004 & 0.0133 & $17.85\pm 0.0096$   \\ \hline
			0.01                 & 0.9821       & 0.0003 & 0.0054 & $37.85\pm 0.0217$   \\ \hline
			0.005               & 0.9850       & 0.0003 & 0.0024 & $70.90\pm 0.0416$  \\ \hline
		\end{tabular}	
		\label{tab:Ex1}
	\end{minipage}\hfill
\end{table}

Here the absolute error $e$ is given by
\begin{equation}
e=|\hat{u}-u|.  \label{eq:error_u}
\end{equation}%
The expected convergence order $O(h)$ can be clearly seen in Figure~\ref%
{fig:Ex1} and Table~\ref{tab:Ex1}.
\end{example}

\subsection{Example with a singular L\'evy measure}

\label{ssec:exa-singular} In this subsection, we confirm dependence of the
error of Algorithm~\ref{Hca01} on the cut-off parameter $\epsilon $ for jump
sizes and on the parameter $\alpha $ of the L\'evy measure as well as
associated computational costs which were derived in Section~\ref{sec:iij}.

\begin{example}[Singular L\'{e}vy measure]
\label{ex:singular-measure} Consider the problem \eqref{eq:problem} with $%
d=3,$ $G=U_{1}$ which is the open unit ball centred at the origin in $%
\mathbb{R}^{3},$ and with the coefficients as in (\ref{ex_ab}), (\ref{ex_F}%
), and%
\begin{align}
g(t,x)& :=\frac{1}{2}e^{T-t}(1.21-x_{1}^{4}-x_{2}^{4})+6(1-\frac{1}{2}%
e^{T-t})\left[ x_{1}^{2}(1.21-x_{2}^{2}-x_{3}^{2})+x_{2}^{2}\right]
\label{ex_g2} \\
& \quad +(1-\frac{1}{2}e^{T-t})\Big[(C_{+}-C_{-})f\left( \frac{4}{\mu }+%
\frac{4}{\mu ^{2}}\right) (x_{1}^{3}+x_{2}^{3})  \notag \\
& \quad +(C_{+}+C_{-})f^{2}\left( \frac{6}{2-\alpha }+\frac{6}{\mu }+\frac{12%
}{\mu ^{2}}+\frac{12}{\mu ^{3}}\right) (x_{1}^{2}+x_{2}^{2})  \notag \\
& \quad +(C_{+}-C_{-})f^{3}\left( \frac{4}{3-\alpha }+\frac{4}{\mu }+\frac{12%
}{\mu ^{2}}+\frac{24}{\mu ^{3}}+\frac{24}{\mu ^{4}}\right) (x_{1}+x_{2})
\notag \\
& \quad +(C_{+}+C_{-})f^{4}\left( \frac{2}{4-\alpha }+\frac{2}{\mu }+\frac{8%
}{\mu ^{2}}+\frac{24}{\mu ^{3}}+\frac{48}{\mu ^{4}}+\frac{48}{\mu ^{5}}%
\right) \Big],  \notag
\end{align}%
with the boundary condition (\ref{ex_fi}), and with the L\'{e}vy measure
density
\begin{equation}
\nu (\mathrm{d}z)=%
\begin{cases}
C_{-}\mathrm{e}^{-\mu (|z|-1)}\mathrm{d}z, & \text{if $z<-1$}, \\
C_{-}|z|^{-(\alpha +1)}\mathrm{d}z, & \text{if $-1\leq z<0$}, \\
C_{+}|z|^{-(\alpha +1)}\mathrm{d}z, & \text{if $0<z\leq 1$}, \\
C_{+}\mathrm{e}^{-\mu (|z|-1)}\mathrm{d}z, & \text{if $z>1$},%
\end{cases}
\label{ex2_de}
\end{equation}%
where $C_{-},$ $C_{+},$ and $\mu $ are some positive constants and $\alpha
\in (0,2)$.

Note that $C_{-}\neq C_{+}$ gives an asymmetric jump measure and the L\'evy
process has infinite activity and variation.

It is not difficult to verify that this problem has the following solution
\begin{equation*}
u(t,x)=(1-\tfrac{1}{2}\mathrm{e}^{T-t})(1.21-x_{1}^{4}-x_{2}^{4}).
\end{equation*}%
Other quantities needed for the algorithm take the form%
\begin{align*}
\gamma _{\epsilon }& =(C_{+}-C_{-})\frac{1-\epsilon ^{1-\alpha }}{1-\alpha }%
,\quad \text{$\alpha \neq 1,$} \\
B_{\epsilon }& =(C_{+}+C_{-})\frac{\epsilon ^{2-\alpha }}{2-\alpha }, \\
\beta _{\epsilon }& =\sqrt{B_{\epsilon }}=\sqrt{(C_{+}+C_{-})\frac{\epsilon
^{2-\alpha }}{2-\alpha }},
\end{align*}%
\begin{align*}
\lambda _{\epsilon }& =\int_{|z|>\epsilon }\nu (\mathrm{d}%
z)=(C_{+}+C_{-})\left( \frac{1}{\mu }+\frac{\epsilon ^{-\alpha }-1}{\alpha }%
\right) , \\
\rho _{\epsilon }(z)& =\frac{1}{\lambda _{\epsilon }}[C_{-}\mathrm{e}^{-\mu
(|z|-1)}\mathbf{I}_{\{z<-1\}}+C_{-}|z|^{-(\alpha +1)}\mathbf{I}_{\{-1\leq
z<-\epsilon \}} \\
& +C_{+}|z|^{-(\alpha +1)}\mathbf{I}_{\{\epsilon <z\leq 1\}}+C_{+}\mathrm{e}%
^{-\mu (|z|-1)}\mathbf{I}_{\{z>1\}}],
\end{align*}%
In this example, the absolute error $e$ is given by
\begin{equation}
e=|\hat{u}^{\epsilon }-u|.  \label{eq:error_ue}
\end{equation}

\begin{figure}[h!]
\centering
\begin{minipage}[t]{.45\textwidth}
  \centering
  \includegraphics[width=\linewidth]{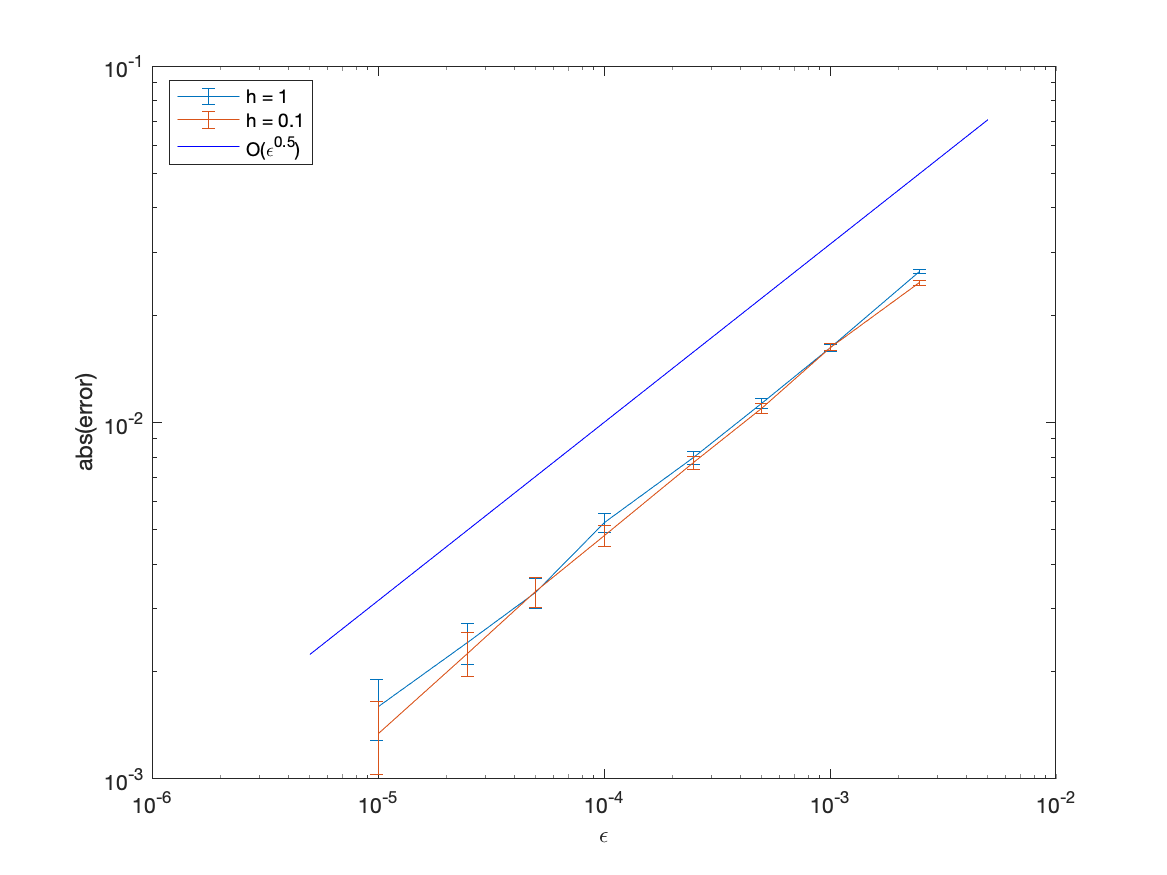}
  \captionof{figure}{Singular L\'evy measure example, the case $\alpha=0.5$: dependence of the error $e$ on $\epsilon$,
the error bars show the Monte Carlo error. The parameters used are $T=1, C_{+}=0.1, C_{-}=1.0, \mu = 3.0, f=0.2, M = 40 000 000$
and $\hat u$ is evaluated at the point $(0,0)$.}
  \label{fig:Ex2-alpha-0.5-eps-error}
\end{minipage}%
\begin{minipage}[t]{.45\textwidth}
  \centering
  \includegraphics[width=\linewidth]{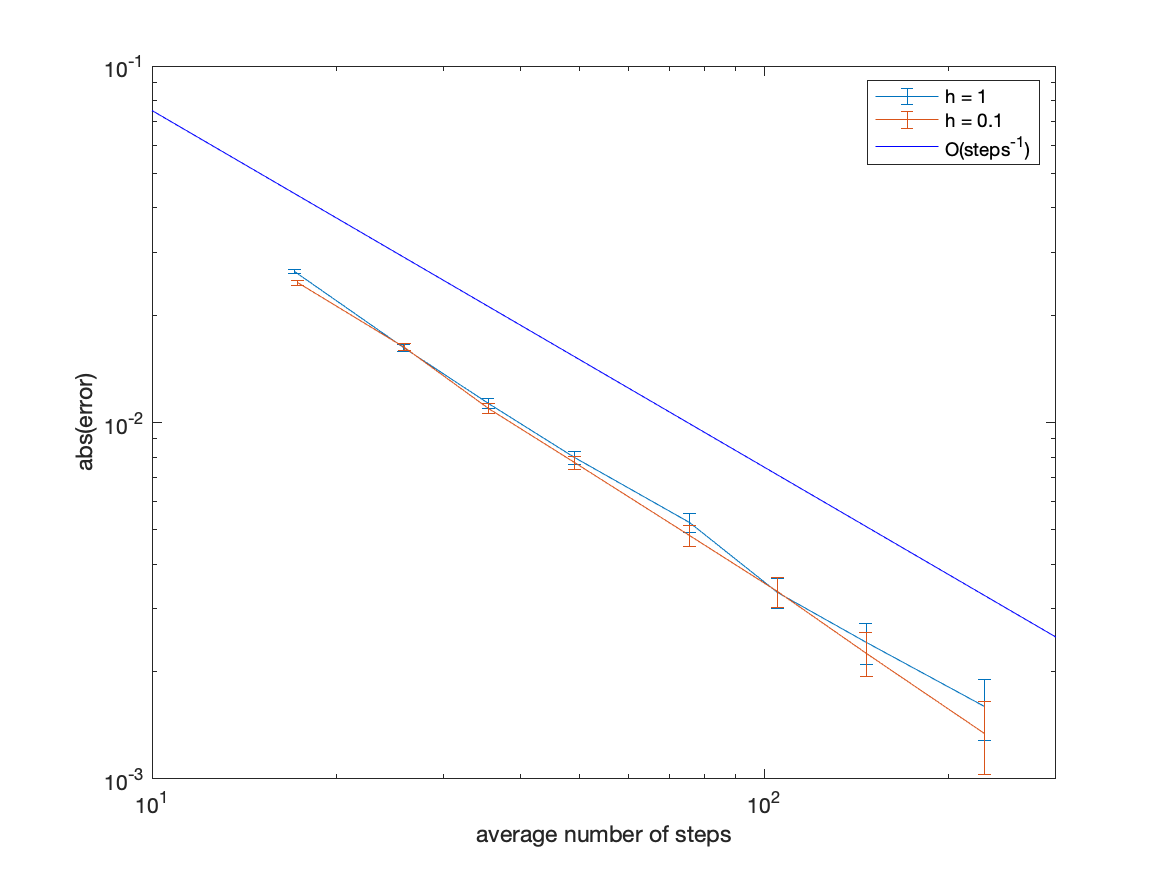}
  \captionof{figure}{Singular L\'evy measure example, the case $\alpha=0.5$: dependence of the error $e$ on the average number of steps (computational costs).
  The parameters are the same as in Figure~\ref{fig:Ex2-alpha-0.5-eps-error}.}
  \label{fig:Ex2-alpha-0.5-steps-error}
\end{minipage}
\end{figure}

\begin{table}[h]
\caption{Singular L\'evy measure example for $\protect\alpha =0.5$ and $h=1$%
. The parameters are the same as in Figure~\protect\ref%
{fig:Ex2-alpha-0.5-eps-error}. The column $\hat{\varkappa}$ gives the sample
average of the number of steps together with its Monte Carlo error.}
\label{tab:Ex2-alpha-0.5}\centering
\begin{tabular}{ccccccc}
\hline
$\epsilon$ & $\hat u$ & $2\sqrt{\hat D_M}$ & $e$ & $\lambda_{\epsilon}$ & $%
\gamma_{\epsilon}$ & $\hat{\varkappa}$ \\ \hline
0.0025 & 0.9610 & 0.0004 & 0.0265 & 42.2 & -1.71 & $17.10\pm 0.0096$ \\
\hline
0.001 & 0.9713 & 0.0004 & 0.0162 & 67.7 & -1.74 & $25.78\pm 0.0149$ \\ \hline
0.0005 & 0.9761 & 0.0004 & 0.0113 & 96.6 & -1.76 & $35.45\pm 0.0208$ \\
\hline
0.00025 & 0.9795 & 0.0003 & 0.0080 & 137.3 & -1.77 & $48.96\pm 0.0290$ \\
\hline
0.0001 & 0.9822 & 0.0003 & 0.0052 & 218.2 & -1.78 & $75.53\pm 0.0452$ \\
\hline
0.00005 & 0.9841 & 0.0003 & 0.0033 & 309.3 & -1.79 & $105.32\pm 0.0633$ \\
\hline
0.000025 & 0.9850 & 0.0003 & 0.0024 & 438.2 & -1.79 & $147.07\pm 0.0888$ \\
\hline
0.00001 & 0.9858 & 0.0003 & 0.0016 & 693.9 & -1.79 & $229.51\pm 0.1393$ \\
\hline
\end{tabular}%
\end{table}
For the case of $\alpha =0.5$, we can clearly see in Figure~\ref%
{fig:Ex2-alpha-0.5-eps-error} and Table~\ref{tab:Ex2-alpha-0.5} that the
error is of order $O(\epsilon ^{\alpha })=O(\epsilon ^{0.5})$ as expected.
We also observe linear convergence in computational cost (measured in
average number of steps). In addition we note that choosing a smaller time
step, e.g. $h=0.1,$ does not change the behaviour in this case which is in
accordance with our prediction of Section~\ref{sec:iij}

\begin{figure}[h]
\centering
\begin{minipage}[t]{.5\textwidth}
  \centering
  \includegraphics[width=1.1\linewidth]{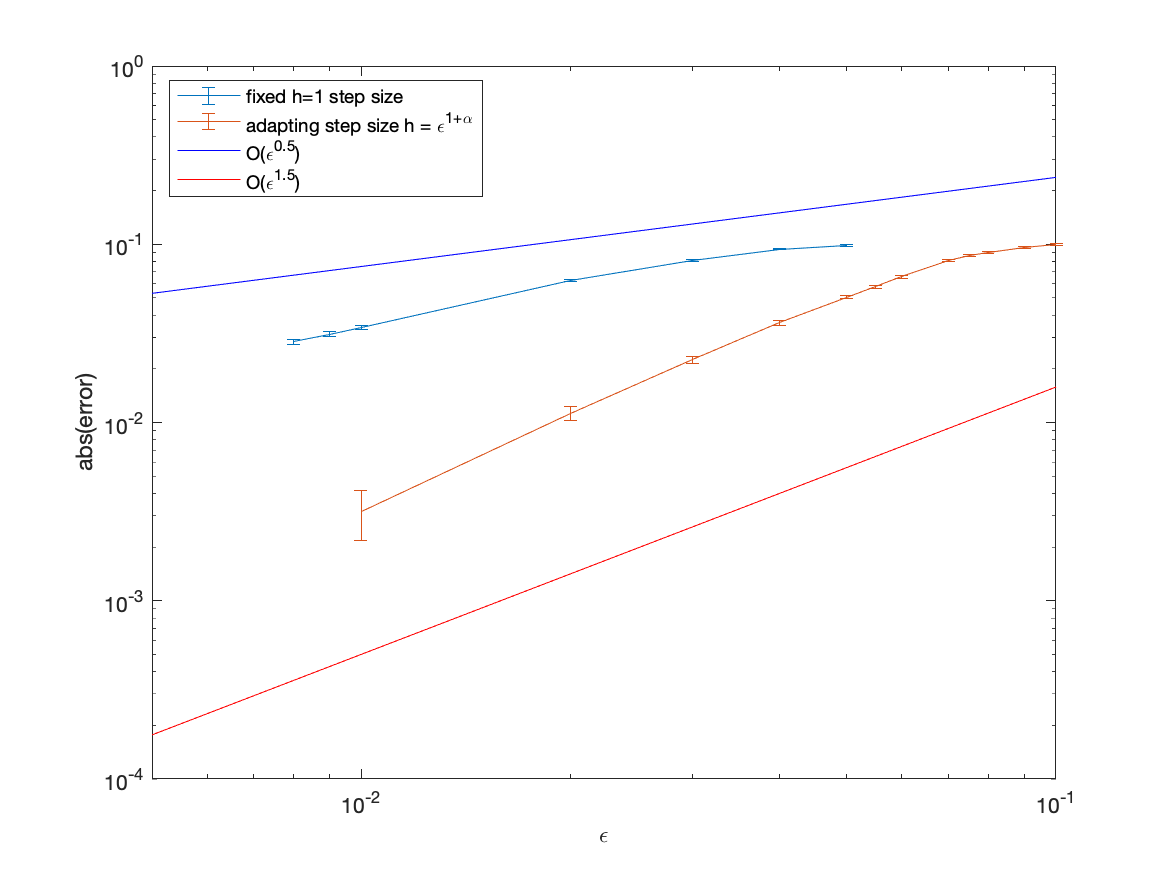}
  \captionof{figure}{Singular L\'evy measure example, the case $\alpha=1.5$: dependence of the error $e$ on $\epsilon$, the error bars show the Monte Carlo error.
  The parameters used are $T=1, C_{+}=1.0, C_{-}=25.0, \mu = 3.0, f=1.0, M = 100 000 000$ and $\hat u$ is evaluated at the point $(0,0)$.}
  \label{fig:Ex5-alpha-1.5-compar-eps-error}
\end{minipage}%
\begin{minipage}[t]{.5\textwidth}
  \centering
 \includegraphics[width=1.1\linewidth]{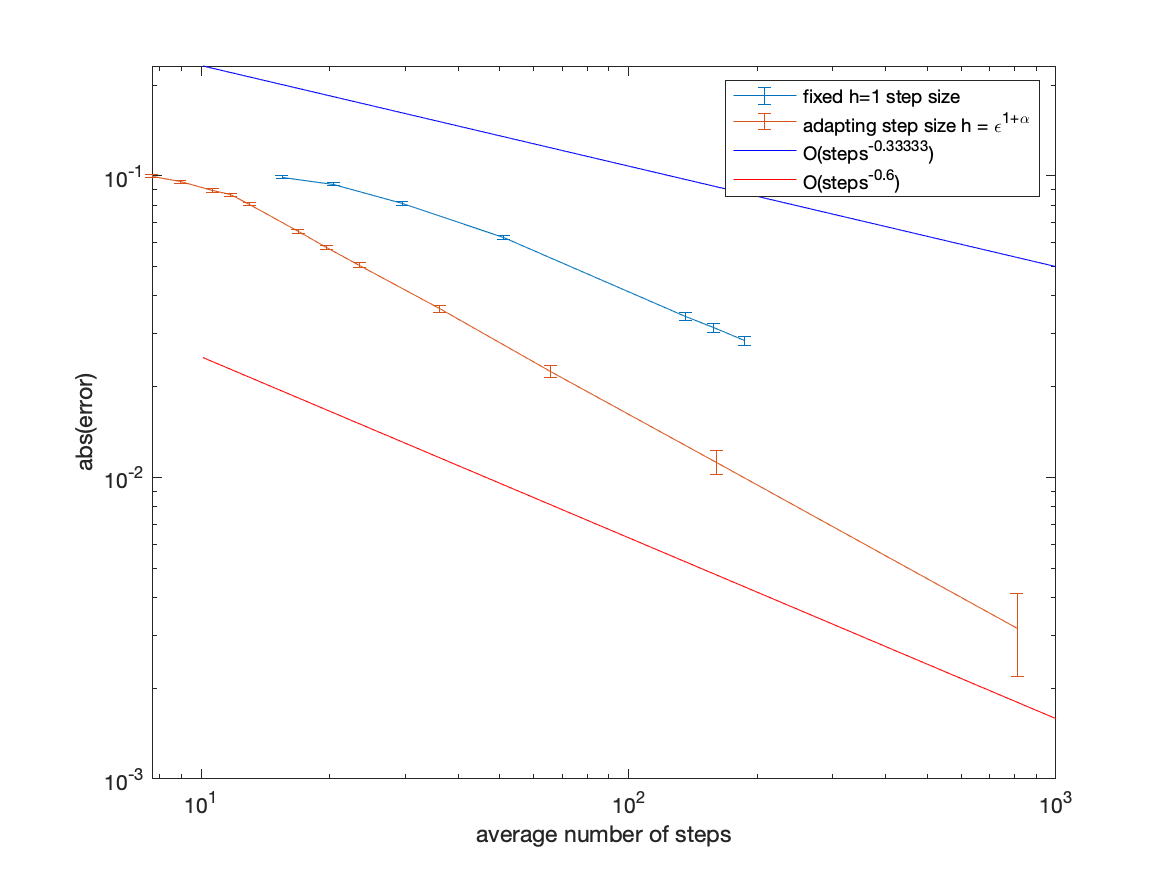}
  \captionof{figure}{Singular L\'evy measure example, the case $\alpha=1.5$: dependence of the error $e$ on the average number of steps (computational costs),
  the error bars show the Monte Carlo error. The parameters are the same as in Figure~\ref{fig:Ex5-alpha-1.5-compar-eps-error}.}
  \label{fig:Ex5-alpha-1.5-compar-steps-error}
\end{minipage}
\end{figure}
Numerical results for the case $\alpha =1.5$ are given in Figures~\ref%
{fig:Ex5-alpha-1.5-compar-eps-error} and \ref%
{fig:Ex5-alpha-1.5-compar-steps-error}. As is shown in Section~\ref{sec:iij}%
, convergence (in terms of computational costs) can be improved in the case
of $\alpha \in (1,2)$ by choosing $h=\epsilon ^{1+\alpha }$. In Figure~\ref%
{fig:Ex5-alpha-1.5-compar-steps-error}, for all $\epsilon $ it can be seen
that choosing a smaller (but optimally chosen) step parameter $h$ results in
quicker convergence (i.e., for the same cost, we can achieve a better result
if $h$ is chosen in an optimal way) and naturally in a smaller error.
\end{example}

We recall that if the jump measure is symmetric, i.e. $C_{-}=C_{+}$ in the
considered example, then $\gamma _{\epsilon }=0$ and the numerical
integration error of Algorithm~\ref{Hca01} is no longer singular (see
Theorem~\ref{thm:total-error} and Remark~\ref{rem:sym}). Consequently (see
Section~\ref{sec:iij}), in this case the computational cost depends linearly
on $\epsilon $ even for $\alpha =1.5,$ which is confirmed on Figure~\ref%
{fig:Ex4-alpha-1.5-eps-error}.
\begin{figure}[h]
\centering
\includegraphics[width=0.55\linewidth]{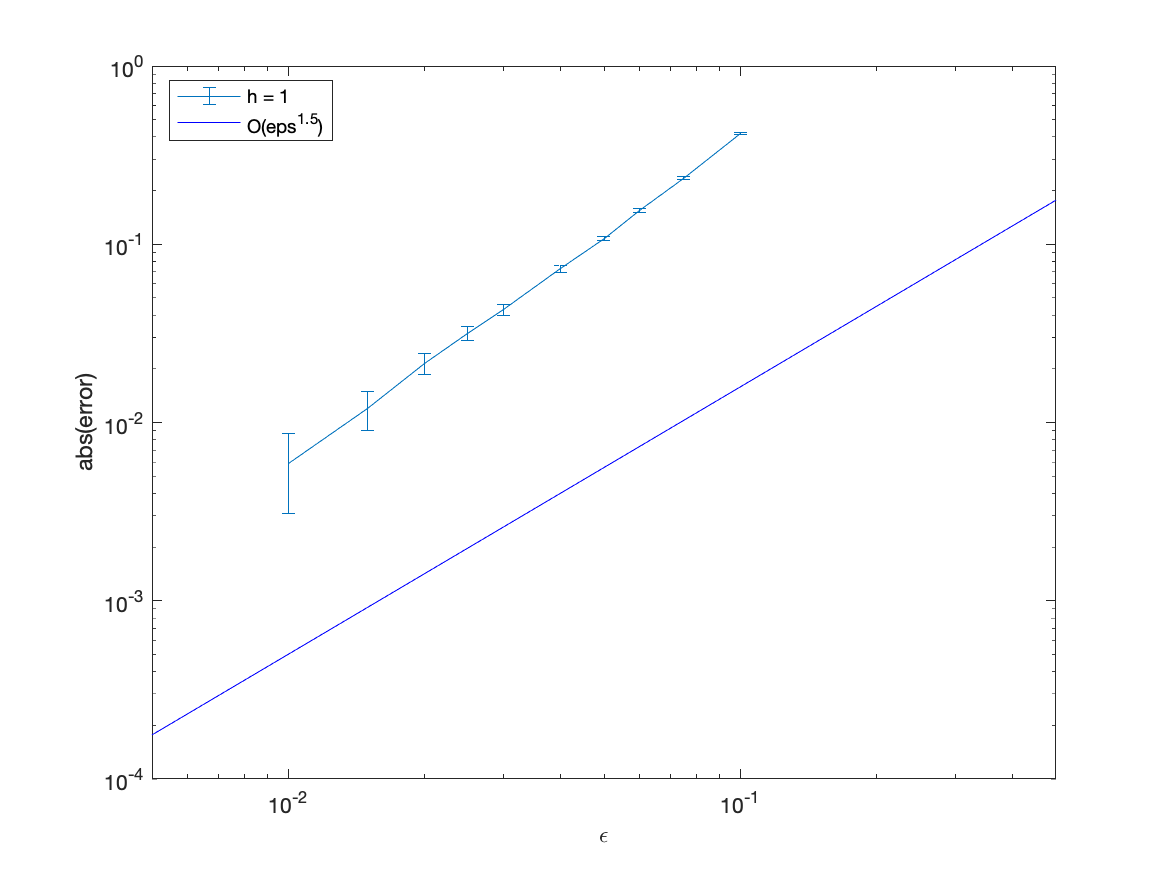}
\captionof{figure}{Dependency of $\epsilon$ on $error$ plot for a simulation
example with symmetric singular L\'evy measure for $\alpha=1.5$. The parameters used are
  $T=1, C_{+}=0.5, C_{-}=0.5, \mu = 3.0, f=1.0, M = 100 000 000$ and $\hat u$ is evaluated at the point $(0,0)$.}
\label{fig:Ex4-alpha-1.5-eps-error}
\end{figure}

\subsection{FX option pricing under a L\'evy-type currency exchange model}

\label{ssec:FX-option-levy}In this subsection, we demonstrate the use of
Algorithm~\ref{Hca01} for pricing financial derivatives where underliers
follow a L\'evy process. We apply the algorithm to estimate the price of a
foreign exchange (FX) barrier basket option. A barrier basket option gives
the holder the right to buy or sell a certain basket of assets (here foreign
currencies) at a specific price $K$ at maturity $T$ in the case when a
certain barrier event has occurred. The most used barrier-type options are
knock-in and knock-out options. This type of option becomes active (or
inactive) in the case of the underlying price $S(t)$ reaching a certain
threshold (the barrier) $B$ before reaching its maturity. In most cases
barrier option prices cannot be given explicitly and therefore have to be
approximated.

\begin{example}[Barrier basket option pricing]
\label{ex:finance} Let us consider the case with five currencies: GBP, USD,
EUR, JPY and CHF and let us assume the domestic currency is GBP. We denote
the corresponding spot exchange rates as
\begin{align*}
S_{1}(t)& =S_{USDGBP}(t), \\
S_{2}(t)& =S_{EURGBP}(t), \\
S_{3}(t)& =S_{JPYGBP}(t), \\
S_{4}(t)& =S_{CHFGBP}(t),
\end{align*}%
where $S_{FORDOM}(t)$ describes the amount of domestic currency DOM one
pays/receives for one unit of foreign currency FOR (for more details see
\cite{Wystup,castagna}). We assume that under a risk-neutral measure $%
\mathbb{Q}$ the dynamics for the spot exchange rates can be written as
\begin{equation*}
S_{i}(t)=S_{i}(t_{0})\exp ((r_{GBP}-r_{i})(t-t_{0})+X_{i}(t)),\qquad
i=1,2,3,4,
\end{equation*}%
where $r_{i}$ are the corresponding short rates of USD, EUR, JPY, CHF and $%
r_{GBP}$ is the short rate for GBP, which are for simplicity assumed to be
constant; and $X(t)$ is a 4-dimensional L\'evy process similar to~%
\eqref{eq:generalsde} with a single jump noise:
\begin{equation}
X(t)=\int\limits_{t_{0}}^{t}b(t,X(s-))ds+\int\limits_{t_{0}}^{t}\sigma
(s,X(s-))dW^{\mathbb{Q}}(t)+\int\limits_{t_{0}}^{t}\int_{\mathbb{R}%
}F(s,S(s-))z\tilde{N}(dz,ds).  \label{FX_X}
\end{equation}%
Here $W(t)=(W_{1}(t),W_{2}(t),W_{3}(t),W_{4}(t))^{\top }$ is a 4-dimensional
standard Wiener process. As $\nu (z),$ we choose the L\'evy measure with
density (\ref{ex2_de}) as in Example~\ref{ex:singular-measure} and we take $%
F(t,x)=(f_{1},f_{2},f_{3},f_{4})^{\top }$ and we will assume that $\sigma
(s,x)$ is a constant $4\times 4$ matrix.

Under the measure $\mathbb{Q}$ all the discounted assets $\hat{S}%
_{i}(t)=e^{-r_{GBP}(t)}S_{i}(t)=S_{i}(t_{0})\exp (-r_{i}(t-t_{0})+X_{i}(t))$
have to be martingales on the domestic market (therefore discounted by the
domestic interest rate) to avoid arbitrage. Using the Ito formula for L\'{e}%
vy processes, we can derive the SDEs for $\tilde{S}_{i}$%
\begin{align*}
\frac{d\tilde{S}_{i}}{\tilde{S}_{i}}& =\left[ -r_{i}+b_{i}(t,X(s-))+\frac{1}{%
2}\sum\limits_{j=1}^{4}\sigma _{ij}+\int\limits_{\mathbb{R}}\left(
e^{f_{i}z}-1-f_{i}z\mathbf{I}_{|z|<1}\right) \nu (dz)\right] dt \\
& \quad +\sum\limits_{j=1}^{4}\sigma _{ij}dW_{j}^{\mathbb{Q}}(t)+\int_{%
\mathbb{R}}f_{i}z\tilde{N}(dz,ds).
\end{align*}%
Hence, for all $\tilde{S}_{i}$ to be martingales, the drift component $b_{i}$
has to be so that
\begin{align*}
b_{i}& =r_{i}-\frac{1}{2}\sum\limits_{j=1}^{4}\sigma _{ij}-\int\limits_{%
\mathbb{R}}\left( e^{f_{i}z}-1-f_{i}z\mathbf{I}_{|z|<1}\right) \nu (dz) \\
& =r_{i}-\frac{1}{2}\sum\limits_{j=1}^{4}\sigma _{ij}-\frac{C_{-}}{\mu +f_{i}%
}e^{-f_{i}}-\frac{C_{+}}{\mu -f_{i}}e^{f_{i}}-\frac{C_{+}-C_{-}}{\mu }%
-I_{i}(\alpha ,C_{+},C_{-}),
\end{align*}%
where
\begin{equation*}
I_{i}(\alpha ,C_{+},C_{-})=\sum\limits_{n=2}^{\infty }\frac{%
(C_{+}+C_{-}(-1)^{n})f_{i}^{n}}{n!(n-\alpha )}.
\end{equation*}%
We also note that
\begin{equation*}
\int\limits_{|z|>1}e^{f_{i}z}\nu (dz)<\infty
\end{equation*}%
is satisfied by (\ref{ex2_de}).

Let us consider an international company based in the UK. If it wants to
protect itself against large FX rate fluctuations, they could hedge their
exposure for each foreign currency on its own. Alternatively, they could use
a knock-in barrier basket option to protect themselves against all the
currency exposure they have, which is in most cases a cheaper way. The value
for such a (down-and-in) put option can be written as
\begin{equation}
P_{t_{0}}(T,K)=\exp ^{-r_{GBP}(T-t_{0})}\mathbb{E}\left[ \mathbf{I}%
_{\min\limits_{t_{0}\leq t\leq T}S(t)<B}\max \left(
K-\sum\limits_{i=1}^{n}w_{i}S_{i}(T),0\right) \right] ,
\label{eq:barrier-price}
\end{equation}%
where $\mathbf{I}_{\min\limits_{t_{0}\leq t\leq T}S(t)<B}=1$ if for any of
the underlying exchange rates $S_{i}(t)<B_{i},\ t_{0}\leq t\leq T$,
otherwise it is zero.

We use Algorithm~\ref{Hca01} together with the Monte Carlo technique to
evaluate this barrier basket option price~\eqref{eq:barrier-price}. In Table~%
\ref{tab:Ex_finance-market-data}, market data for the 4 currency pairs are
given, and in Table~\ref{tab:Ex_finance-option-data} the option and model
parameters are provided, which are used in simulations here.

\begin{table}[h]
\caption{Market data for 4 currency pairs. Here $\protect\sigma _{i}$ are
volatilities for the corresponding pairs and $\protect\rho _{ij}$ are the
correlation coefficients for the corresponding two pairs.}
\label{tab:Ex_finance-market-data}\centering
\begin{tabular}{ccccccc}
\hline
\multicolumn{4}{c}{Market data} & \multicolumn{3}{c}{Correlation data $%
\rho_{ij}$} \\ \hline
currency pair $i$ & $S_i(0)$ & $r_i$ & $\sigma_i$ & USDGBP & EURGBP & JPYGBP
\\ \hline
USDGBP & 0.81 & 0.02 & 0.095 &  &  &  \\ \hline
EURGBP & 0.88 & 0.00 & 0.089 & 0.87 &  &  \\ \hline
JPYGBP & 0.0075 & -0.011 & 0.071 & 0.94 & 0.77 &  \\ \hline
CHFGBP & 0.90 & 0.075 & 0.110 & 0.86 & 0.93 & 0.96 \\ \hline
& $r_{GBP}$ & 0.01 &  &  &  &  \\ \hline
\end{tabular}%
\end{table}
\begin{table}[h]
\caption{Option and model parameters for Example~\protect\ref{ex:finance}}
\label{tab:Ex_finance-option-data}\centering
\begin{tabular}{cccccccc}
\hline
\multicolumn{5}{c}{Option parameter} & \multicolumn{3}{c}{Model parameter}
\\ \hline
currency pair & Barrier $B_i$ & $w_i$ &  &  & jump factor $f_i$ & $\alpha$ &
1.5 \\ \hline
USDGBP & 0.50 & 0.20 & $t_0$ & 0.0 & 0.10 & $C_+$ & 0.3 \\ \hline
EURGBP & 0.60 & 0.25 & $T$ & 1.0 & 0.15 & $C_-$ & 1.2 \\ \hline
JYNGBP & 0.0045 & 0.45 & $K$ & 0.5 & 0.05 & $\mu$ & 3.0 \\ \hline
CHFGBP & 0.55 & 0.10 &  &  & 0.12 & $M$ & $10^6$ \\ \hline
\end{tabular}%
\end{table}

To find the matrix $\sigma =\{\sigma _{ij}\}$ used in the model (\ref{FX_X}%
), we form the matrix $a$ using the volatility $\sigma _{i}$ and correlation
coefficient data from Table~\ref{tab:Ex_finance-market-data} in the usual
way, i.e., $a_{ii}=\sigma _{i}^{2}$ and $a_{ij}=\sigma _{i}\sigma _{j}\rho
_{ij}$ for $i\neq j.$ Then the matrix $\sigma $ is the solution of $\sigma
\sigma ^{\top }=a$ obtained by the Cholesky decomposition.

The results of the simulations are presented in Figure~\ref%
{fig:Ex_finance-compar-eps-error} for different choices of $\epsilon $ and
different choices of $h$. In Figure~\ref{fig:Ex_finance-compar-steps-error},
it can be seen, that (similar to Example~\ref{ex:singular-measure}) by
choosing the step size $h$ optimally results in a better approximation for
the same cost.

\begin{figure}[h]
\centering
\begin{minipage}[t]{.5\textwidth}
  \centering
  \includegraphics[width=1.1\linewidth]{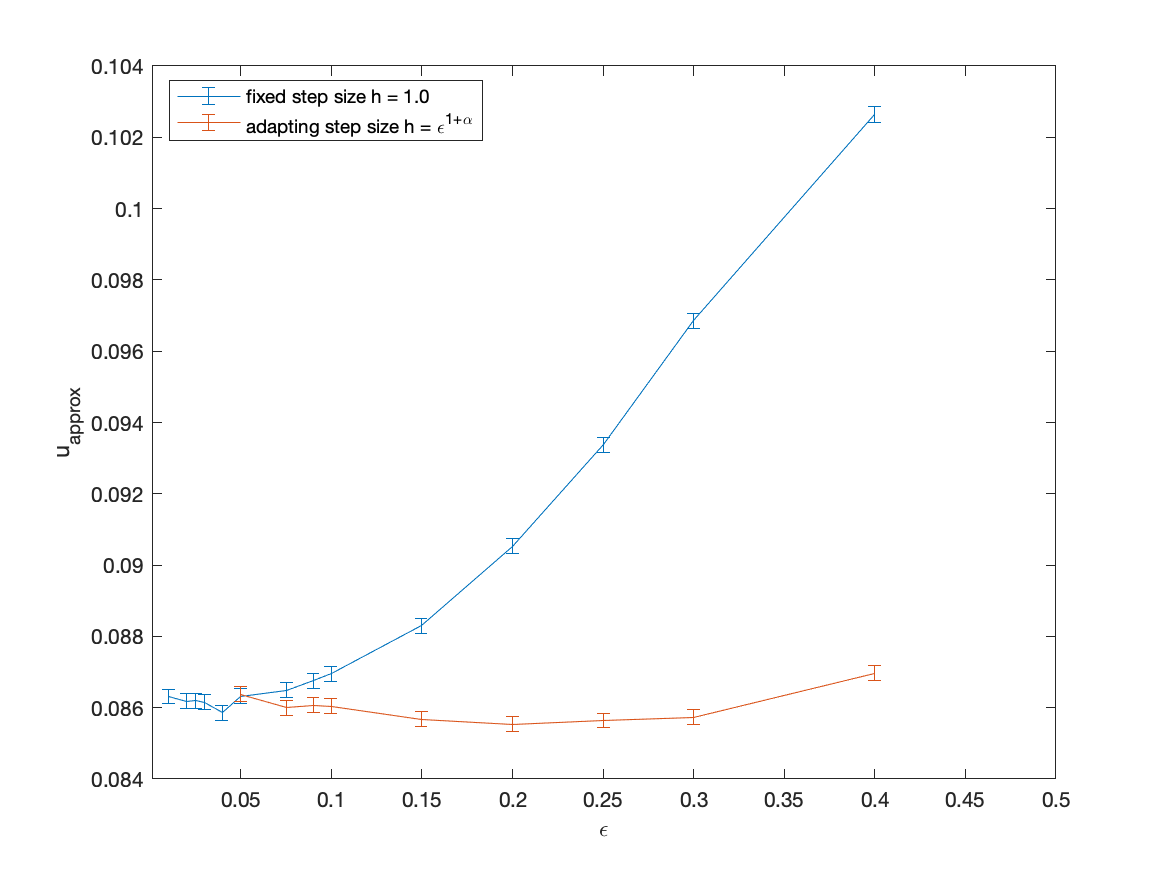}
  \captionof{figure}{Dependence of the approximate price of the FX barrier basket option on $\epsilon$ for different choices of $h$.
  The error bars show the Monte Carlo error.}
  \label{fig:Ex_finance-compar-eps-error}
\end{minipage}%
\begin{minipage}[t]{.5\textwidth}
  \centering
 \includegraphics[width=1.1\linewidth]{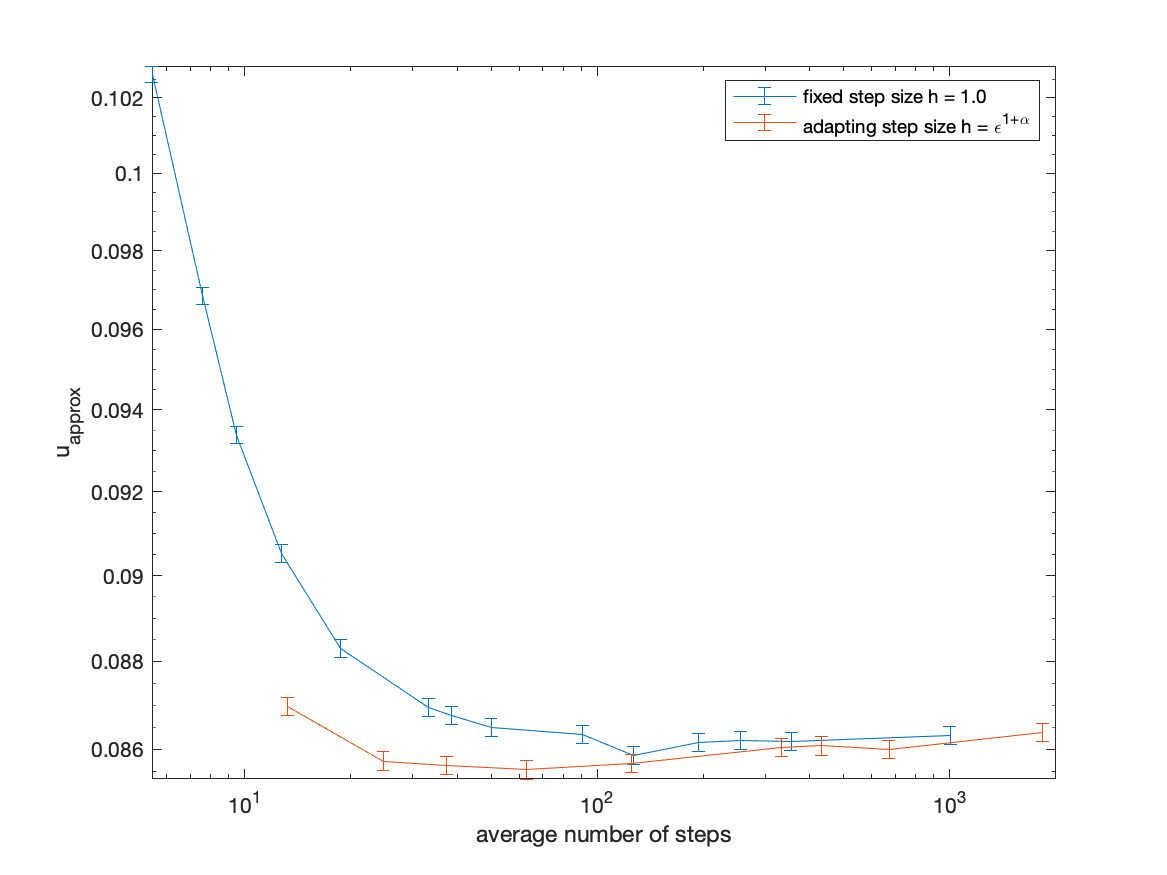}
  \captionof{figure}{Dependence of the approximate price of the FX barrier basket option on average number of steps (computational costs) for different choices of $h$.
  The error bars show the Monte Carlo error.}
  \label{fig:Ex_finance-compar-steps-error}
\end{minipage}
\end{figure}

In this example we demonstrated that Algorithm~~\ref{Hca01} can be
successfully used to price a FX barrier basket option involving 4 currency
pairs following a exponential L\'evy model. In particular, we note that the
algorithm is easy to implement and it gives sufficient accuracy with
relatively small computational costs. Moreover, application of Algorithm~~%
\ref{Hca01} can be easily extended to other multi-dimensional barrier option
(and other types of options and not only on FX markets), while other
approximation techniques such as finite difference methods or Fourier
transform methods typically cannot cope with higher dimensions.
\end{example}

\bibliographystyle{plainnat}
\bibliography{sdeBIB}

\begin{thebibliography}{24}
\providecommand{\natexlab}[1]{#1}
\providecommand{\url}[1]{\texttt{#1}}
\expandafter\ifx\csname urlstyle\endcsname\relax
  \providecommand{\doi}[1]{doi: #1}\else
  \providecommand{\doi}{doi: \begingroup \urlstyle{rm}\Url}\fi

\bibitem[Allen(2003)]{allen2010introduction}
L.~J.~S. Allen.
\newblock \emph{An introduction to stochastic processes with applications to
  biology}.
\newblock CRC Press, 2003.

\bibitem[Applebaum(2009)]{Apple09}
D.~Applebaum.
\newblock \emph{L\'evy processes and stochastic calculus}.
\newblock Cambridge University Press, Cambridge, 2009.

\bibitem[Asmussen and Rosi{\'n}ski(2001)]{Asmus01}
S.~Asmussen and J.~Rosi{\'n}ski.
\newblock Approximations of small jumps of {L}\'evy processes with a view
  towards simulation.
\newblock \emph{J. Appl. Probab.}, 38\penalty0 (2):\penalty0 482--493, 2001.

\bibitem[Barndorff-Nielsen et~al.(2001)Barndorff-Nielsen, Mikosch, and
  Resnick]{barndorff2012levy}
O.~E. Barndorff-Nielsen, T.~Mikosch, and S.~I. Resnick, editors.
\newblock \emph{L{\'e}vy processes: theory and applications}.
\newblock Birkh{\"a}user, 2001.

\bibitem[Castagna(2010)]{castagna}
A.~Castagna.
\newblock \emph{FX options and smile risk}.
\newblock Wiley, 2010.

\bibitem[Cont and Tankov(2004)]{ContTankov}
R.~Cont and P.~Tankov.
\newblock \emph{Financial modelling with jump processes}.
\newblock Chapman \& Hall/CRC, Boca Raton, FL, 2004.

\bibitem[Devroye(1986)]{devroye:1986}
L.~Devroye.
\newblock \emph{Non-uniform random variate generation}.
\newblock Springer, New York, 1986.

\bibitem[Garroni and Menaldi(1992)]{GM92}
M.~G. Garroni and J.-L. Menaldi.
\newblock \emph{Green functions for second order parabolic integro-differential
  problems}, volume 275 of \emph{Pitman Research Notes in Mathematics Series}.
\newblock Longman Scientific \& Technical, Harlow, 1992.

\bibitem[Jacob(2005)]{JacobIII}
N.~Jacob.
\newblock \emph{Pseudo differential operators and {M}arkov processes. {V}ol.
  {III}}.
\newblock Imperial College Press, London, 2005.

\bibitem[Jacod et~al.(2005)Jacod, Kurtz, M{\'e}l{\'e}ard, and
  Protter]{jacod2005approximate}
J.~Jacod, T.~G. Kurtz, S.~M{\'e}l{\'e}ard, and P.~Protter.
\newblock The approximate {E}uler method for {L}{\'e}vy driven stochastic
  differential equations.
\newblock \emph{Annales de l'Institut Henri Poincare (B) Probability and
  Statistics}, 41\penalty0 (3):\penalty0 523--558, 2005.

\bibitem[Kohatsu-Higa and Tankov(2010)]{Higa10}
A.~Kohatsu-Higa and P.~Tankov.
\newblock Jump-adapted discretization schemes for {L}\'evy-driven {SDE}s.
\newblock \emph{Stochastic Process. Appl.}, 120\penalty0 (11):\penalty0
  2258--2285, 2010.

\bibitem[Kohatsu-Higa and Tankov(2013)]{KOT13}
S.~Kohatsu-Higa, A. Ortiz-Latorre and P.~Tankov.
\newblock Optimal simulation schemes for {L}\'evy driven stochastic
  differential equations.
\newblock \emph{Math.Comp.}, 83\penalty0 (289):\penalty0 2293–2324, 2013.

\bibitem[Liu and Li(2000)]{Liu2000}
X.~Q. Liu and C.~W. Li.
\newblock Weak approximation and extrapolations of stochastic differential
  equations with jumps.
\newblock \emph{SIAM J. Numer. Anal.}, 37\penalty0 (6):\penalty0 1747--1767,
  2000.

\bibitem[Mikulevicius and Platen(1988)]{Mik88}
R.~Mikulevicius and E.~Platen.
\newblock Time discrete {T}aylor approximations for {I}to processes with jump
  component.
\newblock \emph{Math. Nachr.}, 138\penalty0 (6):\penalty0 93--104, 1988.

\bibitem[Mikulevicius and Pragarauskas(2005)]{MP05}
R.~Mikulevicius and H.~Pragarauskas.
\newblock On {C}auchy-{D}irichlet problem in half-space for linear
  integro-differential equations in weighted {H}\"older spaces.
\newblock \emph{Electron. J. Probab.}, 10:\penalty0 1398--1416, 2005.

\bibitem[Mikulevicius and Zhang(2018)]{Mik18}
R.~Mikulevicius and C.~Zhang.
\newblock Weak {E}uler scheme for {L}evy-driven stochastic differential
  equations.
\newblock \emph{Theory Probab. Applic.}, 63:\penalty0 346–366, 11 2018.

\bibitem[Milstein and Tretyakov(2002)]{MT03}
G.N. Milstein and M.V. Tretyakov.
\newblock The simplest random walks for the {D}irichlet problem.
\newblock \emph{Theory Probab. Applic.}, 47\penalty0 (1):\penalty0 53--68,
  2002.

\bibitem[Milstein and Tretyakov(2004)]{MT04}
G.N. Milstein and M.V. Tretyakov.
\newblock \emph{Stochastic numerics for mathematical physics}.
\newblock Springer, Berlin, 2004.

\bibitem[Mordecki et~al.(2008)Mordecki, Szepessy, Tempone, and
  Zouraris]{mordecki2008adaptive}
E.~Mordecki, A.~Szepessy, R.~Tempone, and G.~E. Zouraris.
\newblock Adaptive weak approximation of diffusions with jumps.
\newblock \emph{SIAM J. Numer. Anal.}, 46\penalty0 (4):\penalty0 1732--1768,
  2008.

\bibitem[Platen and Bruti-Liberati(2010)]{Platen2010}
E.~Platen and N.~Bruti-Liberati.
\newblock \emph{Numerical solution of stochastic differential equations with
  jumps in finance}.
\newblock Springer, Berlin, 2010.

\bibitem[Protter and Talay(1997)]{Protter97}
P.~Protter and D.~Talay.
\newblock The {E}uler scheme for {L}\'evy driven stochastic differential
  equations.
\newblock \emph{Ann. Probab.}, 25\penalty0 (1):\penalty0 393--423, 1997.

\bibitem[Rubenthaler(2003)]{rubenthaler2003numerical}
S.~Rubenthaler.
\newblock Numerical simulation of the solution of a stochastic differential
  equation driven by a {L}{\'e}vy process.
\newblock \emph{Stoch. Processes Applic.}, 103\penalty0 (2):\penalty0 311--349,
  2003.

\bibitem[van Kampen(2007)]{van1995stochastic}
N.~G. van Kampen.
\newblock \emph{Stochastic processes in physics and chemistry, 3rd edition}.
\newblock North Holland, 2007.

\bibitem[Wystup(2007)]{Wystup}
U.~Wystup.
\newblock \emph{F{X} options and structured products}.
\newblock Wiley, 2007.
\newblock ISBN 9780470057926.

\end{thebibliography}

\appendix

\setcounter{equation}{0} \renewcommand\theequation{A.\arabic{equation}}

\end{document}